\documentclass[11pt]{amsart}
\usepackage{latexsym}
\usepackage{ amssymb }
\usepackage{amsmath}
\usepackage{amsfonts}
\usepackage{amssymb}
\usepackage[all]{xy}
\usepackage{pgf,tikz}
\usepackage{caption}
\usepackage{verbatim}

\usepackage[draft]{todonotes}   

\usetikzlibrary{arrows}
\usetikzlibrary{decorations.pathmorphing}
\theoremstyle{plain}

\newtheorem{theorem}{Theorem}[section]
\newtheorem{lemma}[theorem]{Lemma}
\newtheorem{definition-theorem}[theorem]{Definition-Theorem}
\newtheorem{proposition}[theorem]{Proposition}
\newtheorem{corollary}[theorem]{Corollary}
\newtheorem{definition}[theorem]{Definition}
\newtheorem{conjecture}[theorem]{Conjecture}

\newtheorem*{phenomenon*}{Phenomenon}
\newtheorem*{theorem*}{Theorem}

\theoremstyle{definition}
\newtheorem{example}[theorem]{Example}
\newtheorem{remark}[theorem]{Remark}

\newcommand \bth[1] { \begin{theorem}\label{t#1} }
\newcommand \ble[1] { \begin{lemma}\label{l#1} }

\newcommand \bpr[1] { \begin{proposition}\label{p#1} }
\newcommand \bco[1] { \begin{corollary}\label{c#1} }
\newcommand \bde[1] { \begin{definition}\label{d#1}\rm }
\newcommand \bex[1] { \begin{example}\label{e#1}\rm }
\newcommand \bre[1] { \begin{remark}\label{r#1}\rm }
\newcommand \bcj[1] { \begin{conjecture}\label{j#1}\rm }

\renewcommand {\eth} { \end{theorem} }
\newcommand {\ele} { \end{lemma} }

\newcommand {\epr} { \end{proposition} }
\newcommand {\eco} { \end{corollary} }
\newcommand {\ede} { \end{definition} }
\newcommand {\eex} { \end{example} }
\newcommand {\ere} { \end{remark} }
\newcommand {\ecj} { \end{conjecture} }

\newcommand {\enota} { \end{notation} }

\DeclareMathOperator \green { {\mathrm{green}} }

\DeclareMathOperator \Supp { {\mathrm{Supp}} }

\newcommand \E {\mathbb{E}}

\begin{document}

\title{Building maximal green sequences via component preserving mutations}
\author[Eric Bucher]{Eric Bucher}
\address{
Department of Mathematics \\
Xavier University \\
Cincinnati, OH 45207
U.S.A.
}
\email{buchere1@xavier.edu}

\author[John Machacek]{John Machacek}
\address{
Department of Mathematics and Statistics\\ York  University\\ Toronto, Ontario M3J 1P3\\ CANADA
}
\email{machacek@yorku.ca}

\author[Evan Runburg]{Evan Runburg}
\address{
Department of Mathematics \\
Michigan State University \\
East Lansing, MI 48824
U.S.A.
}
\email{runburge@msu.edu}

\author[Abe Yeck]{Abe Yeck}
\address{
Department of Mathematics \\
Michigan State University \\
East Lansing, MI 48824
U.S.A.
}
\email{yeckmarc@msu.edu}

\author[Ethan Zewde]{Ethan Zewde}
\address{
Department of Mathematics \\
Michigan State University \\
East Lansing, MI 48824
U.S.A.
}
\email{zewdeeth@msu.edu}

\maketitle

\begin{abstract}
We introduce a new method for producing both maximal green and reddening sequences of quivers.
The method, called component preserving mutations, generalizes the notion of direct sums of quivers and can be used as a tool to both recover known reddening sequences as well as find reddening sequences that were previously unknown.
We use the method to produce and recover maximal green sequences for many bipartite recurrent quivers that show up in the study of periodicity of $T$-systems and $Y$-systems.
Additionally, we show how our method relates to the dominance phenomenon recently considered by Reading.
Given a maximal green sequence produced by our method, this relation to dominance gives a maximal green sequence for infinitely many other quivers.
Other applications of this new methodology are explored including computing of quantum dilogarithm identities and determining minimal length maximal green sequences.
\end{abstract}

 \tableofcontents


\section{Introduction}
Quiver mutation is the fundamental combinatorial process which determines the generators and relations in Fomin and Zelevinsky's cluster algebras~\cite{FZ}. Cluster algebras have arisen in a variety of mathematical areas including Poisson geometry, Teichm\"uller theory, applications to mathematical physics, representation theory, and more. 
Quiver mutation is a local procedure that alters a quiver and produces a new quiver.
Understanding how a quiver mutates is essential to understanding the corresponding cluster algebra.
We will consider the problem of explicitly constructing sequences of mutations with some special properties.

\subsection{Some history of the problem}
A maximal green sequence, and more generally a reddening sequence, is a special sequence of quiver mutations related to quantum dilogarithm identities which was introduced by Keller~\cite{KellerQdilog, KellerSquare}.
Such sequences of mutations do not exist for all quivers and determining their existence or nonexistence is an important problem.
For a good introduction to the study of maximal green and reddening sequences see the work of Br\"ustle, Dupont, and P\'erotin~\cite{BDP}. In addition to the role they play in quantum dilogrithm identities, these sequences of mutations are a key tool utilized in other cluster algebra areas.
For example, the existence of a maximal green sequence allows one to categorify the associated cluster algebras following the work of Amiot~\cite{Amiot}. 
Also the existence of a maximal green sequence is a condition which plays a role in the powerful results of Gross, Hacking, Keel, and Kontsevich~\cite{GHKK} regarding canonical bases.
These results use the notion of scattering diagrams to prove the positivity conjecture for a large class of cluster algebras.
Additionally the existence of a reddening sequence is thought to be related to when a cluster algebra equals its upper cluster algebra~\cite{Mills,BanffRed}.
Maximal green sequences are also related to representation theory~\cite{BDP} and in the computation of BPS states in physics~\cite{alim2014}.
Our notion of a \emph{component preserving} sequence of mutations, which will be defined in Section~\ref{sec:CPdef}, is closely related to what has been called a \emph{factorized} sequence of mutations~\cite{Cecotti2011,DelZotto2011,PhysicsA} in the physics literature where particular attention has been paid to $ADE$ Dynkin quivers.
Our definition is more general which allows for use with both maximal green sequences and reddening sequences.
Being able to work with reddening sequences is desirable since the existence of a reddening sequence is mutation invariant while the existence of a maximal green sequence is not~\cite{MullerEJC}.
Hence, the existence of a reddening sequences ends up being a invariant of the cluster algebra as opposed to just the quiver.

In general it can be a difficult problem to determine if a quiver admits a maximal green or reddening sequence. These sequences have been found or shown to not exist in the case of finite mutation type quivers by the work of a variety of authors ~\cite{MillsMGS, Bucher, BucherMills, alim2014, Seven} leaving the question of existence only to quivers that are not of finite mutation type. This makes finding these sequences particularly difficult as the exchange graph for such quivers can be very complicated. Additionally there are branches of the exchange graph, in which no amount of mutations can lead to a maximal green sequence; meaning random computer generated mutations are extremely unlikely to produce maximal green sequences for these quivers. 
In addition to finite mutation type quivers, headway has been made on specific families of quivers such as minimal mutation-infinite quivers~\cite{LawsonMills} and quivers which are associated to reduced plabic graphs~\cite{FS}.
This gives us many quivers for which we know reddening or maximal green sequences for. This provides a foundation to produce reddening and maximal green sequences for quivers which are built out of these.

When a quiver does admit a maximal green or reddening sequence it is desirable to have an explicit and well understood construction of the sequence.
Having the specific sequence of mutations and understanding the corresponding $c$-vectors gives us a product of quantum dilogarithms~\cite{KellerQdilog, KellerSquare} and an expression for the Donaldson-Thomas transformation of Kontsevich and Soibelman~\cite{KS}. The method which we present in this paper allows one to explicitly produce the sequence so that it can be used to for the corresponding computation.

Work by Garver and Musiker~\cite{GM}, as inspired by \cite{Amiot} and \cite{alim2014}, and later by Cao and Li~\cite{CL} looked at using what has been called \emph{direct sums} of quivers to produce maximal green and reddening sequences when the induced subquivers being summed exhibit the appropriate sequences. This heuristic approach of building large sequences of mutations from subquivers is essentially the direction we want to expand upon in this paper. Component preserving mutations are a way of taking known maximal green and reddening sequences for induced subquivers (which we will call components) and combining them together to obtain a maximal green or reddening sequence for the whole quiver.
The direct sum procedure becomes a particular instance of the theory of component preserving mutations. 

The methodology presented has an assortment of applications. It can be used to produce maximal green sequences for bipartite recurrent quivers, recover known results regarding admissible source mutation sequences for acyclic quivers, and show that the existence of a maximal green or reddening sequence is an example of a certain dominance phenomena in the sense of recent work by Reading~\cite{Reading}.

\subsection{Summary of the methodology }\label{sec:methodology}
The goal of this paper is to develop a methodology which allows one to use reddening sequences of subquivers of a given quiver to build reddening sequences for larger quivers. 
Since mutation is a local procedure, only affecting neighboring vertices, this is a natural approach. 
Moreover, it is known that when a quiver has a maximal green or reddening sequence, then the same is true for any induced subquiver~\cite{MullerEJC}.
Hence, developing a method to produce a maximal green or reddening sequence from induced subquivers is a type of converse to this fact.

The method starts by breaking the quiver, $Q$, into subquivers which we call components; each of which has a known reddening sequence. The components will partition the vertices of the quiver, giving a \emph{partitioned quiver} $(Q, \pi)$, where $\pi := \pi_1/\pi_2/\dots/\pi_{ \ell}$ is a partition of the vertices of $Q$. We label the components $Q_i$.  We start with the framed quiver, where we partition all of the frozen and mutable pairs into the same parts. We call this quiver the \emph{framed partition quiver} $(\widehat{Q},\widehat{\pi})$. We then try to \emph{shuffle} the respective reddening sequences together to see if they form a reddening sequence for the entire quiver. 
It is not the case that one can always find a shuffle which works on the entire quiver.
To guarantee that they do build a reddening sequence, we must check that at each mutation step the mutation vertex satisfies the component preserving condition which will be given in Definition~\ref{def:CP}. 
If this condition holds the main result of this paper shows that you have constructed a reddening sequence for the larger quiver. 

\begin{theorem*}[Main Result]
Let $(\widehat{Q},\widehat{\pi})$ be a framed partition quiver where for each $\widehat{Q_i}$ we have a reddening sequence $\sigma_i$. Let $\tau$ be a shuffle of the $\sigma_i$ such that at every mutation step of the sequence $\tau$ we have that $k$ is component preserving with respect to $\pi$. Then $\tau$ is a reddening sequence for $\widehat{Q}$.
\end{theorem*}

This main result is proven in Section~\ref{sec:CPdef} where is it restated in Theorem~\ref{cor:CPred}.
This approach gives one a starting point as to where to search for reddening sequences given an arbitrary quiver. 
First break the quiver into subquivers you are comfortable constructing reddening sequences for; and then attempt to shuffle these sequences. 
This approach may initially seem overwhelming as you could consider any partition of the quiver into subquivers along with any shuffle of reddening sequences.
However, as we explored utilizing this technique what we realized was that there are often very natural shuffles and partitions present in many commonly studied quivers.
For instance, this concept generalizes the idea of direct sums of quivers where the shuffle takes the particular simple form of concatenation.
Additionally, it can be used to give short and effective constructions of maximal green sequences for bipartite recurrent quivers, and many more examples where some \emph{well behaved properties} of a specific quiver provides the recipe for how to shuffle and partition the vertices.

This article is structured in the following way. Section~\ref{sec:prelim} will give some preliminaries for quiver mutation and the study of reddening sequences. In Section~\ref{sec:CPdef} we will present the main results of the paper outlining how the component preserving procedure can produce new maximal green and reddening sequences from induced subquivers. Within Section~\ref{sec:CPdef} we present a large amount of examples to try and illustrate how this procedure works. In the sections following this we look at some applications of this procedure to produce interesting and new results. Results related to dominance phenomena are in Section~\ref{sec:dom} and bipartite recurrent quivers are considered in Section~\ref{sec:bipartite}. In Section~\ref{sec:other} we consider the computation of Donaldson-Thomas invariants and minimal length maximal green sequences.
We have added a large amount of examples to the article in an effort to try and give the reader an opportunity to become familiar with how one uses this method in a hands-on manner. This is intentional, as from exploring these methods it appears that many reddening sequences are built in this manner from small set of ``basic reddening sequences.''
The intuition of the authors is that there may be a way to describe a list of ``basic reddening sequences'' from which any reddening sequence can be built. It is our hope that this paper is the first step in building the concrete theory behind this intuition.

\subsection{Acknowledgements}
The authors would like to thank the anonymous referees for this paper. Their insightful feedback has helped strengthen the paper.

\section{Preliminaries}\label{sec:prelim}
A \emph{quiver} Q is a directed multigraph with vertex set $V(Q)$ and whose edge set $E(Q)$ contains no loops or 2-cycles.
Elements of $E(Q)$ will typically be referred to as \emph{arrows}.
An \emph{ice quiver} is a pair $(Q,F)$ where $Q$ is a quiver, $F \subseteq V(Q)$, and $Q$ contains no arrows between elements of $F$.
Vertices in $F$ are called \emph{frozen} while vertices in $V(Q) \setminus F$ are called \emph{mutable}.
The \emph{framed quiver} associated to a quiver $Q$, denoted $\widehat{Q}$, is the ice quiver whose vertex set, edge set, and set of frozen vertices are the following:
\[ V(\widehat{Q} ) :=  V(Q) \sqcup \{ i' \ | \ i \in V(Q) \} \]
\[ E(\widehat{Q} ) : = E(Q) \sqcup \{ i \rightarrow i' \ | \ i \in V(Q) \} \]
\[F = \{i' | i \in V(Q)\}\]
The framed quiver corresponds to considering a cluster algebra with principal coefficients. 

Given an ice quiver $(Q,F)$ for any mutable vertex $i$, \emph{mutation} at the vertex $i$ produces a new quiver denoted by $(\mu_i(Q),F)$ obtained from $Q$ by doing the following:
\begin{enumerate}
\item[(1)] For each pair of arrows $j \to i$, $i \to k$ such that not both $i$ and $j$ are frozen add an arrow $j \to k$.
\item[(2)] Reverse all arrows incident on $i$.
\item[(3)] Delete a maximal collection of disjoint $2$-cycles.
\end{enumerate}
Mutation is not allowed at any frozen vertex.
Since mutation does not change the set of frozen vertices we will often abbreviate an ice quiver $(Q,F)$ by $Q$ and  $(\mu_i(Q),F)$ by $\mu_i(Q)$ where the set of frozen vertices is understood from context.
We will be primarily focused on framed quivers and quivers which are obtained from a framed quiver by a sequence of mutations.
In fact, whenever we have an ice quiver with a nonempty set of frozen vertices we will assume it is obtainable from a framed quiver by some sequence of mutations.
So, the set of frozen vertices will be of a very particular form.

A mutable vertex is \emph{green} if it there are no incident incoming arrows from frozen vertices.
Similarly, a mutable vertex is \emph{red} if there are no incident outgoing arrows to frozen vertices.  
If we start with an initial quiver $Q$ and perform mutations at mutable vertices of the framed quiver $\widehat{Q}$, then any mutable vertex will always be either green or red.
The result is known as \emph{sign-coherence} and was established by Derksen, Weyman,and Zelevinsky~\cite{DWZ}.
For each vertex $i$ in a quiver obtained from $\widehat{Q}$ by some sequence of mutations, the corresponding \emph{$c$-vector} is defined by its $j$th entry being the number of arrows from $i$ to $j'$ (with arrows $j'$ to $i$ counting as negative).
In these terms sign-coherence says a $c$-vector's entries are either nonnegative or nonpositive.
Notice also that all vertices are initially green when starting with $\widehat{Q}$.
Keller~\cite{KellerQdilog, KellerSquare} has introduced the following types of sequences of mutations which will be our main interest.
A sequence mutations is called a \emph{reddening sequence} if after preforming this sequence of mutations all mutable vertices are red.
A \emph{maximal green sequence} is a reddening sequence where each mutation occurs at a green vertex.
When a sequence of mutations is a reddening sequence we may say it is a reddening sequence for either $Q$ or $\widehat{Q}$.
In terms of being a reddening sequence or not, the quiver $Q$ and the framed quiver $\widehat{Q}$ are equivalent data.

We may write a maximal green or reddening sequence as either a sequence of vertices (read from left to right) or as a composition of mutations (read from right to left as is usual with composition of functions).
For a quiver $Q$ we will let $\green(Q)$ denote the set of maximal green sequences for $Q$.
If we consider the quiver $Q = 1 \to 2$ there are exactly two maximal green sequences and we can record them either as
\[\green(Q) = \{(1,2), (2,1,2)\}\]
in sequence of vertices notation or as
\[\green(Q) = \{\mu_2 \mu_1, \mu_2 \mu_1 \mu_2\}\]
in composition notation.

We will need to modify and combine sequences of vertices when producing maximal green and reddening sequences. This is done by \emph{shuffling} mutation sequences together.

\begin{definition}
A shuffle of two sequences $(a_1,a_2,\dots, a_k)$ and $(b_1,b_2,\dots, b_{\ell})$ is any sequence whose entries are exactly the elements of $\{a_1,a_2,\dots, a_k\} \cup \{b_1,b_2,\dots, b_{\ell}\}$ (considered as a multiset) with the relative orders of $(a_1,a_2,\dots, a_k)$ and $(b_1,b_2,\dots, b_{\ell})$ are preserved. 
\end{definition}

For example there are $6$ shuffles of the sequences $(1,2)$ and $(a,b)$. They are the sequences $(1,2,a,b), \ (1,a,2,b), \ (1,a,b,2), \ (a,1,2,b), \ (a,1,b,2),$ and $(a,b,1,2)$.
In the next section we will define component preserving mutations and show how by checking for the component preserving property you can create shuffles of reddening sequences on induced subquivers whose result is a reddening sequence for a larger quiver.

\section{Component preserving mutations}\label{sec:CPdef}
We start by establishing some basic definitions and notation of what we mean by a component of the quiver.

\begin{definition}
Let $Q$ be an ice quiver with vertex set $V$. Then let $\pi = \pi_1/\pi_2/ \cdots /\pi_{\ell}$ be a set partition of $V$. Then let $Q_i$ be the induced subquiver of $Q$ obtained by deleting every vertex $v \not\in \pi_i$. We will call the $Q_i$ the \textbf{components} of $Q$ and the pair $(Q,\pi)$ a \textbf{partitioned quiver}.
\end{definition}

\begin{definition}
When $(Q, \pi)$ is a partitioned quiver with  $\pi = \pi_1/\pi_2/ \cdots /\pi_{\ell}$, we will define $\widehat{\pi}$ as the partition of $\widehat{V}$ where each $\widehat{\pi_i} = \{v, \widehat{v} \ | \ v \in \pi_i\}$.
Then $(\widehat{Q}, \widehat{\pi})$ will be called a \textbf{partitioned ice quiver}.
\end{definition}

\begin{remark}
In other words, for each mutable vertex $v$, the frozen copy of a vertex, $\widehat{v}$, lies in the same component as $v$. It is straight forward to see that $\widehat{(Q_i)} = (\widehat{Q})_i$.
\end{remark}

\begin{definition}
Mutation of a partitioned ice quiver is defined as the following:
\[\mu_k((Q,\pi)):=(\mu_k(Q),\pi).\]
\end{definition}

\begin{definition}
Let $(Q,\pi)$ be a partitioned ice quiver. A \textbf{bridging arrow} $a\rightarrow b$ is any arrow in $Q$ in which $a$ and $b$ are in different components. 
\end{definition}

Now we can talk about the definition that is crucial to all the results in the rest of the paper. This is the notion of component preserving vertices and component preserving mutations.

\begin{definition}
A vertex $k\in Q_i$ is \textbf{component preserving} with respect to $\pi$ when one of the following occurs:
\begin{itemize}
\item If $\exists \ k\rightarrow j'$ for a frozen vertex $j'$, then $\forall \ a \rightarrow k$ we have $a\in V(Q_i)$; or
\item If $\exists \ j'\rightarrow k$ for a frozen vertex $j'$, then $\forall \ k \rightarrow a$ we have $a\in V(Q_i)$.
\end{itemize}
\label{def:CP}
\end{definition}

\begin{figure}
    \centering
    \begin{tikzpicture}
    \draw (0,0) ellipse (1.5cm and 3cm);
    \node[draw, circle] (k) at (0,0.5) {$k$};
    \node[draw, rectangle] (f) at (0,2) {$j'$};
    \node[draw, circle] (a) at (-0.7,-1.25) {\phantom{$k$}};
    \node[draw, circle] (b) at (0.7,-1.25) {\phantom{$k$}};
    \node[draw, circle] (c) at (2.5,2) {\phantom{$k$}};
    \node[draw, circle] (d) at (2.5,0) {\phantom{$k$}};
    \node[draw, circle] (e) at (2.5,-2) {\phantom{$k$}};
    \draw[-{latex}] (k) to (f);
    \draw[-{latex}] (k) to (a);
    \draw[-{latex}] (b) to (k);
    \draw[-{latex}] (k) to (c);
    \draw[-{latex}] (k) to (d);
    \draw[-{latex}] (k) to (e);
    \node at (0, -3.5) {$Q_i$};
    
    \draw (8,0) ellipse (1.5cm and 3cm);
    \node[draw, circle] (k) at (8,0.5) {$k$};
    \node[draw, rectangle] (f) at (8,2) {$j'$};
    \node[draw, circle] (a) at (7.3,-1.25) {\phantom{$k$}};
    \node[draw, circle] (b) at (8.7,-1.25) {\phantom{$k$}};
    \node[draw, circle] (c) at (10.5,2) {\phantom{$k$}};
    \node[draw, circle] (d) at (10.5,0) {\phantom{$k$}};
    \node[draw, circle] (e) at (10.5,-2) {\phantom{$k$}};
    \draw[-{latex}] (f) to (k);
    \draw[-{latex}] (k) to (a);
    \draw[-{latex}] (b) to (k);
    \draw[-{latex}] (c) to (k);
    \draw[-{latex}] (d) to (k);
    \draw[-{latex}] (e) to (k);
    \node at (8, -3.5) {$Q_i$};
    \end{tikzpicture}
    \caption{An illustration of a component preserving vertex $k \in Q_i$ on the left with arrow $k \to j'$ and on the right with arrow $j' \to k$.}
    \label{fig:my_label}
\end{figure}
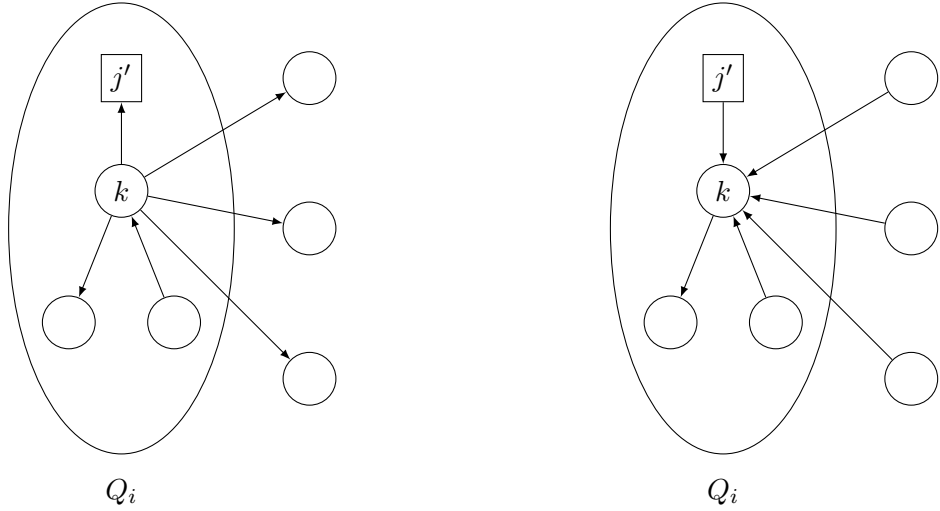

\begin{remark}
Another way of thinking about component preserving mutations is in terms of sign-coherence. One can think of a component preserving vertex, $k$, as a vertex where \emph{freezing} each mutable vertex outside of its component results in an ice quiver in which the extended exchange matrix is still sign-coherent with respect to this larger set of frozen vertices. 
In this way one can think of component preserving mutations as being a type of \emph{locally sign-coherent} mutation. 
\end{remark}

\begin{remark}\label{support} Another observation to make is that whenever one starts from a framed quiver, mutation at component preserving vertices does not result in creating bridging arrows that involve frozen vertices. This means that any quiver which is the result of a sequence of component preserving mutations starting from a framed quiver has the \textbf{support of all of its $c$-vectors contained entirely within a component}. In terms of the quiver, this means that the sequence of component preserving mutations results in a quiver in which all arrows involving frozen vertices are between mutable vertices and frozen vertices within the same component. 
\end{remark}

The choice of terminology is because performing mutation at a component preserving vertex, $k$, does not affect $Q_i$ unless $k\in \pi_i$. We will prove this fact and then show how one can use this fact to shuffle maximal green sequences together if at every mutation step you mutated at a component preserving vertex.

\subsection{Preservation proof}

Now that we have the language to talk about components of the quiver, we want to set up a condition on a vertex, $k$, which forces $\mu_k$ to only affect the component which contains $k$ and none of the other induced subquivers. This is exactly the property that component preserving vertices have.

\begin{lemma}\label{mu}
Let $(Q,\pi)$ be a partitioned ice quiver. If $k$ is a component preserving vertex then $\mu_k(Q)_i=\mu_k(Q_i) \ \forall \ 1 \leq i \leq \ell$. 
\end{lemma}

\begin{proof}
First notice that these are in fact two ice quivers on the same set of vertices. To check that the lemma holds we need to see that each step of mutation has the same effect on the subquivers $\mu_k(Q)_i$ and $\mu_k(Q_i)$ for each $i$. The key step of mutation to check is where new arrows are created, which is step one in our definition of mutation.
There are two cases to consider:

\vspace{.3cm}

\underline{Case One:} $k \in \pi_i$

\vspace{.3cm}

Let $a \rightarrow b$ be an arrow in $\mu_k(Q_i)$ created by mutation at vertex $k$. Then since $Q_i$ is the quiver $Q$ restricted to the component $\pi_i$ we know that $a,b$ along with $k$ are elements of $V(Q_i)$. Therefore the arrows $a \rightarrow k$ and $k \rightarrow b$ are elements of $E(Q_i)$. Therefore all of these arrows are present in $Q$ and hence the arrow $a \rightarrow b$ is present in $\mu_k(Q)$. Since both endpoints of the arrow are in $\pi_i$ the arrow $a\rightarrow b$ is also created in the mutation $\mu_k(Q)_i$. 

We will now show this is a biconditional relationship. Assume $a \rightarrow b$ is an arrow  in $\mu_k(Q)_i$ which is created from mutation.
This occurs if and only if $a \rightarrow k \rightarrow b$ is present in $Q$ and $a,b \in \pi_i$.
Since we have assumed that $k \in \pi_i$ we know that $a,b,k \in \pi_i$ and the arrow $a \to b$ is also created in $\mu_k(Q_i)$.

\vspace{.3cm}

\underline{Case Two:} $k \not\in \pi_i$

\vspace{.3cm}

Since $k$ is not a vertex in $Q_i$ we will not be able to mutate the quiver $Q_i$ in direction $k$. Therefore $\mu_k(Q_i)=Q_i$. Now what we must check is that \emph{no} arrow $a \rightarrow b$ is created in $\mu_k(Q)_i$ by step (1) of mutation. 

By way of contradiction, assume that $a \rightarrow b$ in $\mu_k(Q)_i$ is created by the composition of mutation and restriction.
Then $a \rightarrow k \rightarrow b$ is present in $Q$ and also $a, b \in \pi_i$. 
But since $k$ is not in the same component as $a$ and $b$, arrows $a \rightarrow k$ and $k \rightarrow b$ are bridging arrows in opposite directions. This is a contradiction since each component preserving vertex is incident to bridging arrows in at most one direction.

\end{proof}

\subsection{Applications to reddening sequences and maximal green sequences}

We have seen that if $k$ is a component preserving vertex, then $\mu_k$ only affects arrows in $Q_i$ and possibly bridging arrows.
This can be extremely useful in the context of reddening sequences.
The goal is to utilize reddening sequences on each component to create a reddening sequence for the larger quiver.
This turns out to be possible if at each mutation step you are performing a component preserving mutation.
The following is a useful consequence which follows directly from the sign-coherence of $c$-vectors as presented in \cite{DWZ} and Remark~\ref{support} on the support of $c$-vectors.

\begin{lemma}\label{lemsign}
Let $(\widehat{Q},\widehat{\pi})$ be a partitioned framed quiver. Let $\sigma$ be any sequence of component preserving mutations. Also, let $v$ be a vertex in the component $\pi_i$. Then the color of a vertex $v$ in $\mu_{\sigma}(\widehat{Q})$ is the same as the color of the vertex $v$ in  $\mu_{\sigma}(\widehat{Q})_i$.
\end{lemma}

\begin{theorem}\label{cor1}
Let $(\widehat{Q},\widehat{\pi})$ be a framed partition quiver where for each $\widehat{Q_i}$ we have a reddening sequence $\sigma_i$. Then let $\tau$ be a shuffle of the $\sigma_i$ such that at every mutation step of the sequence $\tau$ we have that $k$ is component preserving with respect to $\pi$. Then $\tau$ is a reddening sequence for $\widehat{Q}$.
\label{cor:CPred}
\end{theorem}

\begin{proof}
Let $(\widehat{Q},\widehat{\pi})$ be a framed partition quiver. Then since each mutation in $\tau$ is component preserving you have from the Lemma~\ref{mu} that
\[ \mu_{\tau}(\widehat{Q})_i=\mu_{\tau}(\widehat{Q_i})=\mu_{\sigma_i}(\widehat{Q_i}).\]

Meaning that for each $i$ any vertex $v \in \pi$ is red in $\mu_{\tau}(\widehat{Q})_i$ since it is the result of running a reddening sequence. 
It then follows from Lemma~\ref{lemsign} that $v$ is red in the larger quiver $\mu_{\tau}(\widehat{Q})$.

\end{proof}

\begin{corollary}\label{cor2}
Furthermore if additionally you have that each $\sigma_i$ is a maximal green sequence for the component $\widehat{Q}_i$ then you have that $\tau$ is a maximal green sequence of for $\widehat{Q}$.
\label{cor:CPMGS}
\end{corollary}

\begin{proof}
By Theorem~\ref{cor:CPred} we know we have a reddening sequence.
By Lemma~\ref{lemsign} and Lemma~\ref{mu} to decide if a mutation step occurred at a green vertex we only need to look at the component containing that vertex. Then we consider that each $\sigma_i$ is a maximal green sequence and it follows from the same equation: \[ \mu_{\tau}(\widehat{Q})_i=\mu_{\tau}(\widehat{Q_i})=\mu_{\sigma_i}(\widehat{Q_i}).\]
\end{proof}

This can be quite useful. In practice what it tells you is that if you partition your quiver up into components, and you know a reddening (or maximal green) sequence for each component then you can try and shuffle the sequences together. If every mutation in the shuffle is component preserving, then you have successfully created a reddening (or maximal green) sequence for the larger quiver.
In the sections that follow we will show some of the applications of using this approach to find maximal green and reddening sequences for a variety of quivers.
Before showing new applications of the component preserving mutation method, we first provide some examples of previously known maximal green sequences that come from component preserving mutations.
These known examples serve to show that our framework unifies many known maximal green sequences.
Also the following examples aim to demonstrate that applications of Corollary~\ref{cor2} occur ``in nature'' and thus Definition~\ref{def:CP} is not too restrictive as it includes many naturally occurring examples. 

\subsection{Example: Admissible source sequences}
\begin{figure}
\begin{tikzpicture}
\node (1) at (0,0) {$1$};
\node (2) at (1,0) {$2$};
\node (3) at (0,1) {$3$};
\node (4) at (-1,0) {$4$};
\node (5) at (0,-1) {$5$};
\draw[-{latex}] (1) to (2);
\draw[-{latex}] (1) to (3);
\draw[-{latex}] (4) to (1);
\draw[-{latex}] (1) to (5);
\end{tikzpicture}
\caption{An acyclic quiver with maximal green sequence $(4,1,2,3,5)$.}
\label{fig:acyclic}
\end{figure}
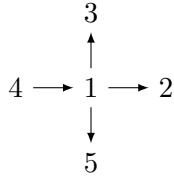

A sequence of vertices $(i_1, i_2, \dots, i_n)$ of a quiver $Q$ with $n$ vertices is called an \emph{admissible numbering} by sources if $\{i_1, i_2, \dots, i_n\} = V(Q)$ and $i_j$ is a source of $\mu_{i_{j-1}} \circ \cdots \circ \mu_{i_1} (Q)$.
It is well known that any acyclic quiver $Q$ admits an admissible numbering by sources and that any such admissible numbering by sources $(i_1, i_2, \dots, i_n)$ is a maximal green sequence~\cite[Lemma 2.20]{BDP}.
In terms of component preserving mutations, $(i_1, i_2, \dots, i_n)$ being an admissible numbering by sources means that $\tau = \mu_{i_n} \circ \mu_{i_{n-1}} \circ \cdots \circ \mu_{i_1}$ is a component preserving sequence of mutations with respect to the partition $\{i_1\}/\{i_2\}/ \cdots /\{i_n\}$ of $V(Q)$ into singletons.
Corollary~\ref{cor:CPMGS} states $(i_1, i_2, \dots, i_n)$ is a maximal green sequence in this special case.
Figure~\ref{fig:acyclic} shows an example of an acyclic quiver where $(4,1,2,3,5)$ is a maximal green sequence from an admissible numbering by sources with the vertices as labeled in the figure.

\subsection{Example: Direct sum}
\begin{figure}
\begin{tikzpicture}
\node (1) at (0,0.5) {$1$};
\node (2) at (0,-0.5) {$2$};
\draw[-{latex}] (1) to (2);

\node (4) at (3,1.5) {$4$};
\node (5) at (3,0) {$5$};
\node (6) at (3,-1.5) {$6$};
\draw[-{latex}] (4) to (5);
\draw[-{latex}] (6) to (5);

\draw[-{latex}] (1) to (4);
\draw[-{latex}] (1) to (5);
\draw[-{latex}] (2) to (4);
\draw[-{latex}] (2) to[bend right=30] (5);
\draw[-{latex}] (2) to (5);
\draw[-{latex}] (2) to (6);
\end{tikzpicture}
\caption{A direct sum of quivers with maximal green sequence $(2,1,2,4,6,5)$.}
\label{fig:directsum}
\end{figure}
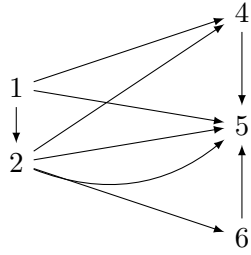

A \emph{direct sum} of quivers $A$ and $B$ is any quiver $Q$ with
\[V(Q) = V(A) \sqcup V(B)\]
\[E(Q) = E(A) \sqcup E(B) \sqcup E\]
where $E$ is any set of arrows such which has for any $i \to j \in E$ implies $i \in V(A)$ and $j \in V(B)$.
In other words, a direct sum of quivers simply takes the disjoint union of the two quivers then adds additional arrows between the quivers with the condition that all arrows are directed from one quiver to the other.
We can take the partition $V(A)/V(B)$ of $V(Q)$ and the consider the concatenation $\tau = \tau_B \tau_A$ for any reddening sequence $\tau_A$ of $A$ and $\tau_B$ of $B$.
Then $\tau$ will be component preserving and hence a reddening sequence by Theorem~\ref{cor:CPred}

An example of a direct sum of quivers $A$ and $B$ where $V(A) = \{1,2\}$ and $V(B) = \{4,5,6\}$ is given in Figure~\ref{fig:directsum}.
We can take the maximal green sequences $(2,1,2)$ and $(4,6,5)$ on the components and obtain maximal green sequence $(2,1,2,4,6,5)$ on the direct sum.
We will not prove that such sequences of mutations are component preserving since proofs for maximal green sequences and reddening sequences of direct sums are already in the literature~\cite[Theorem 3.12]{GM}~\cite[Theorem 4.5]{CL}.

\subsection{Example: Square products}

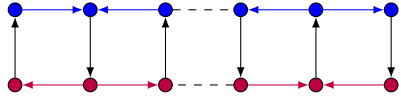
\begin{figure}

\begin{tikzpicture}
\node[draw, circle,fill=blue,scale=0.5] (1) at (-1,1) {};
\node[draw, circle,fill=purple,scale=0.5] (2) at (-1,0) {};
\node[draw, circle,fill=blue,scale=0.5] (4) at (0,1) {};
\node[draw, circle,fill=purple,scale=0.5] (5) at (0,0) {};
\node[draw, circle,fill=blue,scale=0.5] (7) at (1,1) {};
\node[draw,  circle,fill=purple,scale=0.5] (8) at (1,0) {};
\node[draw, circle,fill=blue,scale=0.5] (10) at (2,1) {};
\node[draw, circle,fill=purple,scale=0.5] (11) at (2,0) {};
\node[draw, circle,fill=blue,scale=0.5] (12) at (3,1) {};
\node[draw, circle,fill=purple,scale=0.5] (13) at (3,0) {};
\node[draw, circle,fill=blue,scale=0.5] (14) at (4,1) {};
\node[draw, circle,fill=purple,scale=0.5] (15) at (4,0) {};

\draw[-{latex},blue] (1) to (4);
\draw[-{latex},blue] (7) to (4);
\draw[-{latex},blue] (12) to (10);
\draw[-{latex},blue] (12) to (14);
\draw[dashed] (7) to (10);
\draw[-{latex},purple] (5) to (2);
\draw[-{latex},purple] (5) to (8);
\draw[-{latex},purple] (11) to (13);
\draw[-{latex},purple] (15) to (13);
\draw[dashed] (11) to (8);
\draw[-{latex}] (2) to (1);
\draw[-{latex}] (4) to (5);
\draw[-{latex}] (8) to (7);
\draw[-{latex}] (10) to (11);
\draw[-{latex}] (13)to (12);
\draw[-{latex}] (14) to (15);
\end{tikzpicture}
\caption{An arbitrary length square product of type $(A_2, A_n)$.}
\label{fig:squareA2}
\end{figure}

The \emph{square product} of two Dynkin quivers is considered by Keller in his work on periodicity~\cite{KellerPeriod}.
For two type $A$ quivers the square product is a grid with all square faces oriented in a directed cycle. 
In Figure~\ref{fig:squareA2} we show a square product of type ($A_2$, $A_n$).
Consider the partition $\pi=B/B'$ of the quiver in Figure~\ref{fig:squareA2} where $B$ is the set of vertices in the top row and $B'$ is the set of vertices in the bottom row.
Then the quiver restricted to either $B$ or $B'$ is an alternating path which has a maximal green sequence of repeatedly applying sink mutations. 
A component preserving shuffle for these quivers can be found by alternating between mutations in $B$ and $B'$ until you have completed both maximal green sequences.
This example generalizes to many other quivers in a family called \emph{bipartite recurrent quivers}.
Maximal green sequences for bipartite recurrent quivers will be investigated in more depth in Section~\ref{sec:bipartite}.

\subsection{Example: Dreaded torus}\label{sec:torus}
\begin{figure}
\begin{tikzpicture}
\node (1) at (0,0) {$1$};
\node (2) at (2,0) {$2$};
\node (3) at (1,0.75) {$3$};
\node (4) at (1,2) {$4$};
\draw[-{latex}] (1) to (2);
\draw[-{latex}] (1) to (3);
\draw[-{latex}] (2) to (3);
\draw[-{latex}] (4) to (1);
\draw[-{latex}] (4) to (2);
\draw[-{latex}, bend right=13] (3) to (4);
\draw[-{latex},bend left=13] (3) to (4);
\end{tikzpicture}
\caption{The quiver for the torus with one boundary component and one marked point. A maximal green sequence for this quiver is $(1,3,4,2,1,3)$.}
\label{fig:dreaded}
\end{figure}
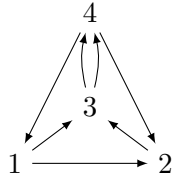

Let $Q$ be the quiver shown in Figure~\ref{fig:dreaded} which comes from a triangulation of the torus with one boundary component and a single marked point on the boundary.
With vertices as labeled in the figure we can take the partition $\{1,4\}/\{2,3\}$  and the maximal green sequences $(1,4,1)$ and $(3,2,3)$ on the two components.
The sequence $(1,3,4,2,1,3)$ is component preserving and hence a maximal green sequence by Corollary~\ref{cor:CPMGS}.
The quiver $Q$ is an example of a quiver which admits a maximal green sequence, and hence a reddening sequence, but is not a member of the class $\mathcal{P}$ of Kontsevich and Soibelman~\cite{KS}.
So, $Q$ should be included in a solution to a question posed by the first two authors which seeks to identify a collection of quivers which generate all quivers with reddening sequences by using quiver mutation and the direct sum construction~\cite[Question 3.6]{BanffRed}.

\subsection{Example: Cremmer-Gervais}

In the Gekhtman, Shapiro, and Vainshtein approach to cluster algebras with Poisson geometry there is an exotic cluster structure on $SL_n$ known as the Cremmer-Gervais cluster structure~\cite{GSV-CGshort, GSV-CGlong}.
The mutable part of the quiver defining this cluster structure for the case $n = 3$ is shown in Figure~\ref{fig:CG}.
The cluster algebra has the interesting property that whether or not it agrees with its upper cluster algebra is ground ring dependent~\cite[Proposition 4.1]{BMS}.
A maximal green sequence for the quiver in Figure~\ref{fig:CG} is $(2,3,4,1,5,1,6,3)$ which can be obtained by considering the partition $\{1,2,5\} / \{3,6\} / \{4\}$ along with maximal green sequences $(2,1,5,1)$, $(3,6,3)$, and $(4)$.
The authors believe it would be interesting to try the technique of component preserving maximal green sequences on quivers for the Cremmer-Gervais cluster structure for larger values for $n$.

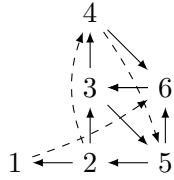
\begin{figure}
\begin{tikzpicture}
\node at (2,3)  (23) {$4$};
\node at (2,2)  (22) {$3$};
\node at (3,2)  (32) {$6$};
\node at (1,1)  (11) {$1$};
\node at (2,1)  (21) {$2$};
\node at (3,1)  (31) {$5$};

\draw[-{latex}] (21) -- (11);
\draw[-{latex}] (21) -- (22);
\draw[-{latex}] (32) -- (22);
\draw[-{latex}] (22) -- (31);
\draw[-{latex}] (31) -- (32);
\draw[-{latex}] (31) -- (21);
\draw[-{latex}] (22) -- (23);
\draw[-{latex}] (23) -- (32);
\draw[-{latex}] (32) -- (22);

\draw[-{latex}, dashed] (11) [bend right=8]to (32);
\draw[-{latex}, dashed] (21) [bend left=20]to (23);
\draw[-{latex}, dashed] (23) [bend left=8]to (31);

\end{tikzpicture}
\caption{The mutable part of the quiver defining the Cremmer-Gervais cluster structure.}
\label{fig:CG}
\end{figure}

\section{Applications to quiver dominance}\label{sec:dom}

One natural question that arises when discussing any algebraic object is to ask questions about what information can be extracted from considering the smaller sub-objects inside your larger object.
The methods we have presented thus far give a way of producing reddening sequences on larger quivers by considering reddening sequences on quivers with fewer vertices.
In this section we will give a way of producing reddening sequences on larger quivers by considering reddening sequences on quivers with fewer arrows but the same number of vertices.

Component preserving mutations give rise to a \emph{dominance phenomenon} of quivers. 
In terms of matrices dominance is given by the following definition.
One obtains a definition of dominance in quivers by considering its skew-symmetric exchange matrix.

\begin{definition}
Given $n \times n$ exchange matrices $B = [b_{ij} ]$ and $A = [a_{ij}]$, we say $B$ dominates $A$ if for each $i$ and $j$, we have $b_{ij}a_{ij} \geq 0$ and $|b_{ij} | \geq |a_{ij}|$. 
\end{definition}

An initiation of a systematic study of dominance for exchange matrices was put forth by Reading~\cite{Reading}. 
Dominance had previously been considered by Huang, Li, and Yang~\cite{seedHom} as part of their definition of a \emph{seed homomorphism}.
One instance of the dominance phenomenon observed by Reading is the following observation about scattering fans.

\begin{phenomenon*}[{\cite[Phenomenon III]{Reading}}]
Suppose that $B$ and $B'$ are exchange matrices such that $B$ dominates $B'$. In many cases, the scattering fan of $B$ refines the scattering fan of $B'$.
\end{phenomenon*}

\begin{remark}
Following~\cite{GHKK} to any quiver one can associate a cluster scattering diagram inside some ambient vector space.
Reddening sequences and maximal green sequences then correspond to paths in the ambient vector space subject to certain restrictions coming from the scattering diagram.
A cluster scattering diagram partitions the ambient vector into a complete fan called the scattering fan~\cite{scattering}.
Hence, the phenomenon that the scattering fan of $B$ often refines the scattering fan of $B'$ when $B$ dominates $B'$ means that it should be more difficult to find a reddening sequence for $B$ since the scattering diagram of $B$ has additional walls imposing more constraints.
However, we will find certain conditions for when a reddening sequence for $B'$ will still work as a reddening sequence for $B$.
\end{remark}


In this section we will apply the results of Section~\ref{sec:CPdef} to show that the existence of a reddening (maximal green) sequence passes through the dominance relationship in many cases. The interesting aspect of this result is it appears to go in the wrong direction; the property is passed from the dominated quiver to the dominating quiver.
Let $B$ dominate $A$. If $A$ has a reddening (maximal green) sequence then, we wish to produce a reddening (maximal green) sequence for $B$. This is not a true statement in general, but if we put some restrictions on \emph{how} $B$ dominates $A$ and \emph{extra conditions} on the reddening or maximal green sequence this turns out to be true.
Going forward we will consider dominance in terms of the quivers instead of exchange matrices. A reformulation of dominance is the following.

\begin{definition}
Given quivers $B$ and $A$ on the same vertex set we say that $B$ dominates $A$ if:

\begin{itemize}
    \item for every pair of vertices $(i,j)$ any arrows between $i$ and $j$ in $A$ are in the same direction as any arrow between $i$ and $j$ in $B$; and
    \item for every pair of vertices $(i,j)$ the number of arrows in $B$ involving vertices $i$ and $j$ is greater than or equal to the number of arrows in $A$ involving $i$ and $j$.
\end{itemize}
\end{definition}

\begin{figure}
\begin{tikzpicture}
\node[draw, circle,fill=black,scale=0.5] (1) at (-1,1) {};
\node[draw, circle,fill=black,scale=0.5] (2) at (-1,0) {};
\node[draw, circle,fill=black,scale=0.5] (4) at (0,1) {};
\node[draw, circle,fill=black,scale=0.5] (5) at (0,0) {};
\node[draw, circle,fill=black,scale=0.5] (7) at (1,1) {};
\node[draw, circle,fill=black,scale=0.5] (8) at (1,0) {};
\node[draw, circle,fill=black,scale=0.5] (10) at (2,1) {};
\node[draw, circle,fill=black,scale=0.5] (11) at (2,0) {};

\draw[-{latex}] (1) to (4);
\draw[-{latex}] (7) to (4);
\draw[-{latex}] (7) to (10);
\draw[-{latex}] (5) to (2);
\draw[-{latex}] (5) to (8);
\draw[-{latex}] (11) to (8);
\draw[-{latex}] (2) to (1);
\draw[-{latex}] (4) to (5);
\draw[-{latex}] (8) to (7);
\draw[-{latex}] (10) to (11);

\node[draw, circle,fill=black,scale=0.5] (19) at (5,1) {};
\node[draw, circle,fill=black,scale=0.5] (12) at (5,0) {};
\node[draw, circle,fill=black,scale=0.5] (14) at (6,1) {};
\node[draw, circle,fill=black,scale=0.5] (15) at (6,0) {};
\node[draw, circle,fill=black,scale=0.5] (17) at (7,1) {};
\node[draw, circle,fill=black,scale=0.5] (18) at (7,0) {};
\node[draw, circle,fill=black,scale=0.5] (110) at (8,1) {};
\node[draw, circle,fill=black,scale=0.5] (111) at (8,0) {};

\node[scale=.5] at (4.8,.5) {2};
\node[scale=.5] at (5.8,.5) {3};
\node[scale=.5] at (6.8,.5) {4};
\node[scale=.5] at (7.8,.5) {5};

\draw[-{latex}] (19) to (14);
\draw[-{latex}] (17) to (14);
\draw[-{latex}] (17) to (110);
\draw[-{latex}] (15) to (12);
\draw[-{latex}] (15) to (18);
\draw[-{latex}] (111) to (18);

\draw[-{latex}] (12) to (19);
\draw[-{latex}] (14) to (15);
\draw[-{latex}] (18) to (17);
\draw[-{latex}] (110) to (111);

\end{tikzpicture}
\caption{An example where the quiver on the right dominates the quiver on the left.}
\label{fig:dominance}
\end{figure}
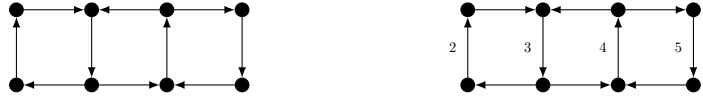

For an example of quiver dominance see Figure~\ref{fig:dominance} where multiplicity of an arrow greater than $1$ is denoted by the number next to the arrow.
We now need to establish the notion of \emph{$\pi$-dominance}.
This is a restrictive form of dominance, where we the quivers $A$ and $B$ have the same component subquivers with respect to a partition $\pi$ but have the multiplicity of the bridging arrows altered in a consistent way. 

\begin{definition}
Let $(A, \pi)$ and $(B, \pi)$ be two partitioned ice quivers with the same vertex set and same set partition $\pi$. We say that $B$ $\pi$-dominates \footnote{This is a more restrictive version of the dominance phenomena presented by Reading. In general, not all quivers $B$ which dominate a quiver $A$ will $\pi$-dominate the quiver.} $A$ if:

\begin{itemize}
    \item the component quivers $A_i=B_i$ for each $i$;
    \item for all $u\in B_i$ and $v \in B_j$ with $i\neq j$ we have the $\# (u \rightarrow v$ in $B)$ is equal to $d_{ij} \times \#(u \rightarrow v$ in $A)$, where $d_{ij}$ is a positive integer that is the same for the entire $i$-th and $j$-th components.
\end{itemize}
\label{def:pidom}
\end{definition}

The $d_{ij}$ are called the \emph{dominance constants} associated to $(B, \pi)$ and $(A, \pi)$.
As usual in Definition~\ref{def:pidom} arrows in the opposite direction are counted as negative.
A practical way of thinking about $\pi$-dominance is that $B$ is obtained from the $A$ by scaling up the multiplicity of the bridging arrows between components by the appropriate dominance constant.
Notice that the dominance constants are always positive, and hence bridging arrows are always in the same direction after scaling by the dominance constants.
An example of $\pi$-dominance can be seen in Figure \ref{fig:pidominance}.
This example has the type $(A_2, A_4)$ square product on the left side and the $Q$-system quiver of type $A_4$ on the right side.  

\begin{figure}
\begin{tikzpicture}
\node[draw, circle,fill=blue,scale=0.5] (1) at (-1,1) {};
\node[draw, circle,fill=purple,scale=0.5] (2) at (-1,0) {};
\node[draw, circle,fill=blue,scale=0.5] (4) at (0,1) {};
\node[draw, circle,fill=purple,scale=0.5] (5) at (0,0) {};
\node[draw, circle,fill=blue,scale=0.5] (7) at (1,1) {};
\node[draw,  circle,fill=purple,scale=0.5] (8) at (1,0) {};
\node[draw, circle,fill=blue,scale=0.5] (10) at (2,1) {};
\node[draw, circle,fill=purple,scale=0.5] (11) at (2,0) {};

\draw[-{latex},blue] (1) to (4);
\draw[-{latex},blue] (7) to (4);
\draw[-{latex},blue] (7) to (10);
\draw[-{latex},purple] (5) to (2);
\draw[-{latex},purple] (5) to (8);
\draw[-{latex},purple] (11) to (8);
\draw[-{latex}] (2) to (1);
\draw[-{latex}] (4) to (5);
\draw[-{latex}] (8) to (7);
\draw[-{latex}] (10) to (11);

\node[draw, circle,fill=blue,scale=0.5] (19) at (5,1) {};
\node[draw, circle,fill=purple,scale=0.5] (12) at (5,0) {};
\node[draw, circle,fill=blue,scale=0.5] (14) at (6,1) {};
\node[draw, circle,fill=purple,scale=0.5] (15) at (6,0) {};
\node[draw, circle,fill=blue,scale=0.5] (17) at (7,1) {};
\node[draw, circle,fill=purple,scale=0.5] (18) at (7,0) {};
\node[draw, circle,fill=blue,scale=0.5] (110) at (8,1) {};
\node[draw, circle,fill=purple,scale=0.5] (111) at (8,0) {};

\node[scale=.5] at (4.8,.5) {2};
\node[scale=.5] at (5.8,.5) {2};
\node[scale=.5] at (6.8,.5) {2};
\node[scale=.5] at (7.8,.5) {2};

\draw[-{latex},blue] (19) to (14);
\draw[-{latex},blue] (17) to (14);
\draw[-{latex},blue] (17) to (110);
\draw[-{latex},purple] (15) to (12);
\draw[-{latex},purple] (15) to (18);
\draw[-{latex},purple] (111) to (18);

\draw[-{latex}] (12) to (19);
\draw[-{latex}] (14) to (15);
\draw[-{latex}] (18) to (17);
\draw[-{latex}] (110) to (111);

\end{tikzpicture}
\caption{This is $\pi$-dominance where the components are the horizontal rows of the quiver. The right hand quiver $\pi$-dominates the left hand quiver and $d_{12}=2$.}
\label{fig:pidominance}
\end{figure}
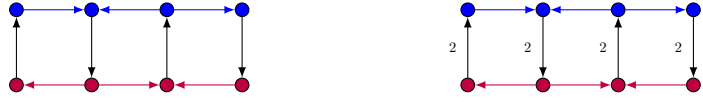

\begin{theorem}\label{dominancetheorem}
Let $k$ be a component preserving vertex in $(A,\pi)$ and $(B, \pi)$ be an ice quiver which $\pi$-dominates $A$ with dominance constants $d_{ij}$. Then $\mu_k(B)$ dominates $\mu_k(A)$ with dominance constants $d_{ij}$. 
\end{theorem}

\begin{proof}
Since $k$ is a component preserving vertex in $(A,\pi)$ we know that $k$ is also a component preserving vertex in $(B,\pi)$ since the direction of the bridging arrows is unchanged by scaling by the multiple $d_{ij}$. Also as $k$ is component preserving in both $A$ and $B$ we know by Lemma~\ref{mu} that $\mu_k(A)_i=\mu_k(A_i)=\mu_k(B_i)=\mu_k(B)_i$. Therefore we only need to consider the bridging arrows between components. 

The bridging arrows incident to $k$ are only affected by the step of mutation which reverses arrows incident to $k$. Therefore dominance is preserved for these arrows because they are reversed by mutation at $k$ in both $A$ and $B$. 

Now we must check the number of bridging arrows created during mutation for both $\mu_k(B)$ and $\mu_k(A)$. 
For some nonnegative integer $\alpha$, we will use the notation $i \stackrel{\alpha}{\rightarrow} j$ to denote that there are $\alpha$ arrows from $i$ to $j$ in a quiver.

Assume $s\stackrel{\alpha}{\rightarrow} k \stackrel{\beta}{\rightarrow} t$ is present in $A$ with $\alpha, \beta \geq 0$. 
Then mutation will create arrows from $s \rightarrow t$ with multiplicity $\alpha \beta$.
Since we need only consider bridging arrows we will assume the $\alpha \beta$ many arrows from $s$ to $t$ created are bridging arrows.
In the case that $k$ is green we know that $s$ must be in the same component as $k$ because $k$ is component preserving.
Assume $k, s \in V(A_i)$ and $t \in V(A_j)$ for $i \neq j$.
We now will show that $\mu_k(B)$ creates $d_{ij}\alpha \beta$ arrows from $s$ to $t$.
The presence of $s\stackrel{\alpha}{\rightarrow} k \stackrel{\beta}{\rightarrow} t$ in $A$ implies that there is $s \stackrel{\alpha}{\rightarrow} k \stackrel{d_{ij}\beta}{\rightarrow} t$ in $B$.
Therefore mutation at $k$ in $B$ creates $d_{ij}\alpha \beta $ arrows $s\rightarrow t$.
Now we can consider the multiplicity of bridging arrows resulting from cancellation of $2$-cycles mutation.
In $\mu_k(A)$ the multiplicity of the arrows from $s$ to $t$ is $\alpha \beta + \gamma$, where $\gamma$ is the number of arrows from $s$ to $t$ in $A$ (here we allow $\gamma$ to be negative if there are arrows from $t$ to $s$). 
In $\mu_k(B)$ the multiplicity of arrows from $s$ to $t$ is $d_{ij} \alpha \beta + d_{ij}\gamma $ since there are $d_{ij} \gamma$ arrows from $s$ to $t$ in $B$ by the assumption that $B$ $\pi$-dominates $A$. 
Therefore there are exactly $d_{ij}(\alpha \beta + \gamma)$ arrows from $s$ to $t$ in $\mu_k(B)$ which is exactly the condition needed to say that $\mu_k(B)$ $\pi$-dominates $\mu_k(A)$.

The case where $k$ is red is very similar.
In this case $t$ must be in the same component as $k$ because $k$ is component preserving.
The presence of $s\stackrel{\alpha}{\rightarrow} k \stackrel{\beta}{\rightarrow} t$ in $A$ now implies that there is $s \stackrel{d_{ij}\alpha}{\rightarrow} k \stackrel{\beta}{\rightarrow} t$ in $B$.
Again mutation at $k$ in $B$ creates $d_{ij}\alpha \beta $ arrows $s\rightarrow t$ and the rest of the argument follows the case where $k$ was green.
\end{proof}

We can now state our main result regarding dominance, that certain reddening sequences can be passed from a quiver $A$ to a $\pi$-dominating quiver $B$.

\begin{corollary}\label{dominancecor}
Let $(A,\pi)$ be a partitioned quiver, with $\pi=\pi_1/\pi_2/\dots / \pi_{\ell}$. Let $\sigma_1,\sigma_2, \dots , \sigma_{\ell}$ be reddening sequences for $A_1,A_2, \dots , A_{\ell}$ respectively. If $A$ admits a reddening sequence, $\tau$, which is a component preserving shuffle of $\sigma_1,\sigma_2, \dots , \sigma_{\ell}$ and $B$ $\pi$-dominates $A$, then $\tau$ is also a reddening sequence for $B$. Moreover, if $\tau$ is a maximal green sequence for $A$, then $\tau$ is a maximal green sequence for $B$. 
\end{corollary}

\begin{proof}
Theorem \ref{dominancetheorem} shows that each component preserving mutation in $A$ is also a component preserving mutation in $B$. Therefore the mutation sequence $\tau$ is a component preserving sequence for $B$ since it is a component preserving sequence for $A$. The definition of $\pi$-dominance tells us that $A_1=B_1,A_2=B_2,\dots , A_{\ell}=B_{\ell}$. Therefore since $\sigma_1, \sigma_2,\dots \sigma_{\ell}$ are reddening sequences for $A_1,A_2,\dots , A_{\ell}$, they are also reddening sequences for $B_1,B_2,\dots , B_{\ell}$. Then by Theorem~\ref{cor1} and Corollary~\ref{cor2} we have that they are in fact reddening sequences and additionally maximal green in the case where each $\sigma_i$ is a maximal green sequence. 
\end{proof}

\begin{figure}
    \begin{tikzpicture}
    \node (1) at (1,1) {$1$};
    \node (2) at (2,0) {$2$};
    \node (3) at (1,-1) {$3$};
    \node (4) at (-1,-1) {$4$};
    \node (5) at (-2,0) {$5$};
    \node (6) at (-1,1) {$6$};
    
    \draw[->] (1) to (2);
    \draw[->] (2) to (3);
    \draw[->] (3) to (4);
    \draw[->] (4) to (5);
    \draw[->] (5) to (6);
    \draw[->] (6) to (1);
    
    \node at (-1.75,-0.65) {$d$};
    \node at (-1.75,0.65) {$d$};
    \end{tikzpicture}
    \caption{A quiver dominating the cycle which has the maximal green sequence $(1,2,3,4,5,6,4,3,2,1)$.}
    \label{fig:dom_cyle}
\end{figure}
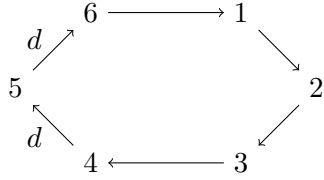

Now we are equipped to use $\pi$-dominance to produce reddening and maximal green sequences for the dominating quivers by having \emph{well behaved} sequences on the dominated quiver.
We conclude this section with a few examples each providing a family of applications of Corollary~\ref{dominancecor}.

\subsection{Examples of applying Corollary~\ref{dominancecor}}

Corollary~\ref{dominancecor} applies to any case where one can produce a maximal green or reddening seqeunce using component preserving mutations.
Thus, this result can be applied in many cases to produce infinite families of examples.
In this section we highlight a few examples.

\begin{example}[Dreaded Torus]
Previously much attention has been paid to maximal green sequences for finite mutation type quivers (see~\cite{MillsMGS}).
In Section~\ref{sec:torus} we saw one example of a maximal green sequence for a finite mutation type quiver using component preserving mutations.
Now we revisit this example, except we can scale the bridging arrows between the components and leave the case of finite mutation type.
By Corollary \ref{dominancecor} we know that the original maximal green sequence for the dreaded torus will also be a maximal green sequence for all $\pi$-dominating quivers. Therefore $(1,3,4,2,1,3)$ is a maximal green sequence for all of the quivers in Figure \ref{fig:dreaded-dominance}, where $a$ is a positive integer. This is an example of a quiver where the shuffle is not one that can be obtained from direct sum results as the partition does not form a direct sum of either the original quiver or the $\pi$-dominating quivers.

\begin{figure}
\begin{tikzpicture}
\node (1) at (0,0) {$1$};
\node (a) at (2.55,1.6) {$a$};
\node (b) at (1.2,1.5) {$2a$};
\node (c) at (.8,.6) {$a$};
\node (d) at (1.5,-.3) {$a$};
\node (2) at (3,0) {$2$};
\node (3) at (1.5,0.75) {$3$};
\node (4) at (1.5,3) {$4$};
\draw[-{latex}] (1) to (2);
\draw[-{latex}] (1) to (3);
\draw[-{latex}] (2) to (3);
\draw[-{latex}] (4) to (1);
\draw[-{latex}] (4) to (2);
\draw[-{latex}] (3) to (4);
\end{tikzpicture}
\caption{For each positive integer $a$, Corollary \ref{dominancecor} produces a maximal green sequence for the quiver, which was the maximal green sequence from the dreaded torus. The maximal green sequence is $(1,3,4,2,1,3)$.}
\label{fig:dreaded-dominance}
\end{figure}
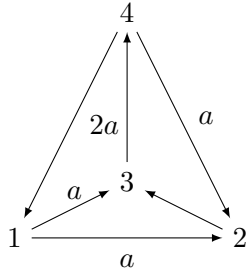
\end{example}

\begin{example}[The cycle]
Another example of finite mutation type quiver is the directed cycle quiver with vertex set $\{1,2,\dots, n\}$ and arrow set $\{i \to (i+1) : 1 \leq i < n\} \cup \{n \to 1\}$.
In~\cite[Lemma 4.2]{Bucher} it is shown this quiver has the maximal green sequence 
\[(1,2,\dots,n-2,n-1,n,n-2,n-3,\dots,2,1)\]
which can be seen to be component preserving with respect to the partition $\{1,2,\dots,n-3,n-2,n\}/\{n-1\}$.
By applying Corollary~\ref{dominancecor} we then obtain maximal green sequences for many quivers of infinite mutation type.
The case $n=6$ is shown in Figure~\ref{fig:dom_cyle}.
\end{example}

\begin{example}[$Q$-systems]
Consider Figure~\ref{fig:pidominancemgs} when $\alpha =2$ in which we can produce a maximal green sequence for the $Q$-system quiver of type $A_4$ by utilizing the maximal green sequence from the square product quiver of type $(A_2, A_4)$.
This technique also produces maximal green sequences for other $Q$-system quivers (see~\cite{Q1,Q2}) which are dominating quivers of square products.
The next section will focus on producing maximal green sequences for a variety of bipartite recurrent quivers.
\end{example}

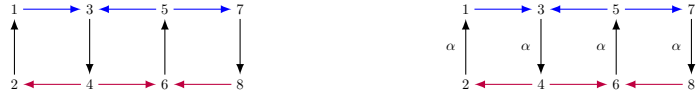
\begin{figure}
\begin{tikzpicture}
\node[scale=0.5] (1) at (-1,1) {1};
\node[scale=0.5] (2) at (-1,0) {2};
\node[scale=0.5] (4) at (0,1) {3};
\node[scale=0.5] (5) at (0,0) {4};
\node[scale=0.5] (7) at (1,1) {5};
\node[scale=0.5] (8) at (1,0) {6};
\node[scale=0.5] (10) at (2,1) {7};
\node[scale=0.5] (11) at (2,0) {8};

\draw[-{latex},blue] (1) to (4);
\draw[-{latex},blue] (7) to (4);
\draw[-{latex},blue] (7) to (10);
\draw[-{latex},purple] (5) to (2);
\draw[-{latex},purple] (5) to (8);
\draw[-{latex},purple] (11) to (8);
\draw[-{latex}] (2) to (1);
\draw[-{latex}] (4) to (5);
\draw[-{latex}] (8) to (7);
\draw[-{latex}] (10) to (11);

\node[scale=0.5] (19) at (5,1) {1};
\node[scale=0.5] (12) at (5,0) {2};
\node[scale=0.5] (14) at (6,1) {3};
\node[scale=0.5] (15) at (6,0) {4};
\node[scale=0.5] (17) at (7,1) {5};
\node[scale=0.5] (18) at (7,0) {6};
\node[scale=0.5] (110) at (8,1) {7};
\node[scale=0.5] (111) at (8,0) {8};

\node[scale=.5] at (4.8,.5) {$\alpha$};
\node[scale=.5] at (5.8,.5) {$\alpha$};
\node[scale=.5] at (6.8,.5) {$\alpha$};
\node[scale=.5] at (7.8,.5) {$\alpha$};

\draw[-{latex},blue] (19) to (14);
\draw[-{latex},blue] (17) to (14);
\draw[-{latex},blue] (17) to (110);
\draw[-{latex},purple] (15) to (12);
\draw[-{latex},purple] (15) to (18);
\draw[-{latex},purple] (111) to (18);

\draw[-{latex}] (12) to (19);
\draw[-{latex}] (14) to (15);
\draw[-{latex}] (18) to (17);
\draw[-{latex}] (110) to (111);

\end{tikzpicture}
\caption{This is $\pi$-dominance where the components are the horizontal rows of the quiver. The square product quiver on the left has a maximal green sequence compatible with a $\pi$ component preserving shuffle of $(2,3,6,7,1,4,5,8,2,3,6,7,1,4,5,8,2,3,6,7)$. Corollary \ref{dominancecor} shows that the quiver on the left where $\alpha$ is any positive integer admits the same maximal green sequence.}
\label{fig:pidominancemgs}
\end{figure}

\section{Bipartite recurrent quivers}\label{sec:bipartite}

In this section we consider certain quivers arising in the setting of $T$-systems and $Y$-systems.
An early application of cluster algebras was Fomin and Zelevinsky's proof of periodicity for $Y$-systems associated to root systems~\cite{FZYsystem} which was conjectured by Zamolodchikov~\cite{Zam}.
This has lead to many more applications of cluster algebra theory in periodicity for $T$-systems and $Y$-systems.
We will focus on work of Galashin and Pylyavskyy on bipartite recurrent quivers~\cite{GP1,GP2,GP3}.
For certain bipartite recurrent quivers we will produce maximal green sequences in Theorem~\ref{thm:recurrentMGS}.
An important ingredient in our constructions of maximal green sequences will be an extension of Stembridge's bigraphs~\cite{Stem}.
The pattern for the maximal green sequences produced in this section was originally observed by Keller in the case of square products~\cite{KellerPeriod}.
For a quantum field theory perspective on the results in this section we refer the reader to~\cite{PhysicsA} where some of the same mutation sequences we construct are also considered.
The main contribution of this section is to demonstrate how component preserving mutation neatly establishes the existence of a maximal green sequence for all quivers in Galashin and Pylyavskyy's classification of Zamolodchikov periodic quivers~\cite{GP1} as well as for some additional bipartite recurrent quivers.

We call a quiver $Q$ \emph{bipartite} if there exists a map $\epsilon: V(Q) \to \{0,1\}$ such that $\epsilon(i) \neq \epsilon(j)$ for every arrow $i \to j$ of $Q$.
The choice of such a map $\epsilon$ when it exists for a quiver $Q$ is called a \emph{bipartition}.
Given a bipartition $\epsilon$ for $Q$ a vertex $i \in V(Q)$ will be called \emph{white} if $\epsilon(i) = 0$ and $\emph{black}$ if $\epsilon(i) = 1$.
Let $i_1, i_2, \dots, i_{\ell}$ denote the white vertices and $Q$ and $j_1, j_2, \dots, j_m$ denote the black vertices.
We then let 
\[\mu_{\circ} = \mu_{i_1} \circ \mu_{i_2} \circ \cdots \circ \mu_{i_{\ell}}\]
and
\[\mu_{\bullet} = \mu_{j_1} \circ \mu_{j_2} \circ \cdots \circ \mu_{j_m}\]
denote the mutations at all white vertices or black vertices respectively.
Since the quiver is bipartite no white vertex is adjacent to any other white vertex and so the order of mutation among the white vertices in $\mu_{\circ}$ does not matter.
Similarly the order among the black vertices in $\mu_{\bullet}$ does not matter.
A bipartite quiver $Q$ is \emph{recurrent} if both $\mu_{\circ}(Q) = Q^{op}$ and $\mu_{\bullet}(Q) = Q^{op}$ where $Q^{op}$ denotes the quiver obtained from $Q$ by reserving the direction of all arrows.
Thus for a bipartite recurrent quiver we have $\mu_{\bullet}(\mu_{\circ}(Q)) = Q$ and $\mu_{\circ}(\mu_{\bullet}(Q)) = Q$.

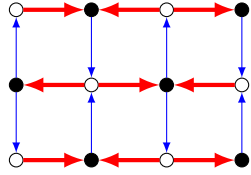
\begin{figure}
\begin{tikzpicture}
\node[draw, circle,fill=white,scale=0.5] (1) at (-1,1) {};
\node[draw, circle,fill=black,scale=0.5] (2) at (-1,0) {};
\node[draw, circle,fill=white,scale=0.5] (3) at (-1,-1) {};
\node[draw, circle,fill=black,scale=0.5] (4) at (0,1) {};
\node[draw, circle,fill=white,scale=0.5] (5) at (0,0) {};
\node[draw, circle,fill=black,scale=0.5] (6) at (0,-1) {};
\node[draw, circle,fill=white,scale=0.5] (7) at (1,1) {};
\node[draw, circle,fill=black,scale=0.5] (8) at (1,0) {};
\node[draw, circle,fill=white,scale=0.5] (9) at (1,-1) {};
\node[draw, circle,fill=black,scale=0.5] (10) at (2,1) {};
\node[draw, circle,fill=white,scale=0.5] (11) at (2,0) {};
\node[draw, circle,fill=black,scale=0.5] (12) at (2,-1) {};

\draw[-{latex},red,ultra thick] (1) to (4);
\draw[-{latex},red,ultra thick] (7) to (4);
\draw[-{latex},red,ultra thick] (7) to (10);
\draw[-{latex},red,ultra thick] (5) to (2);
\draw[-{latex},red,ultra thick] (5) to (8);
\draw[-{latex},red,ultra thick] (11) to (8);
\draw[-{latex},red,ultra thick] (3) to (6);
\draw[-{latex},red,ultra thick] (9) to (6);
\draw[-{latex},red,ultra thick] (9) to (12);

\draw[-{latex},blue] (2) to (3);
\draw[-{latex},blue] (2) to (1);
\draw[-{latex},blue] (4) to (5);
\draw[-{latex},blue] (6) to (5);
\draw[-{latex},blue] (8) to (7);
\draw[-{latex},blue] (8) to (9);
\draw[-{latex},blue] (10) to (11);
\draw[-{latex},blue] (12) to (11);
\end{tikzpicture}
\caption{An example of a bipartite recurrent quiver.}
\label{fig:square}
\end{figure}

A \emph{bigraph} is a pair $(\Gamma, \Delta)$ of undirected graphs on the same underlying vertex set with no edges in common.
Let $A_{\Gamma}$ and $A_{\Delta}$ denote the adjacency matrices of $\Gamma$ and $\Delta$ respectively.
Given any bipartite quiver $Q$ with bipartition $\epsilon$ we obtain a bigraph $(\Gamma(Q), \Delta(Q))$ on vertex set $V(Q)$ where $\Gamma(Q)$ has an edge $\{i,j\}$ for each arrow $i \to j$ in $Q$ with $\epsilon(i) = 0$ and $\Delta(Q)$ has an edge $\{i,j\}$ for each arrow $i \to j$ of $Q$ with $\epsilon(i) = 1$.
By abuse of notation we may also think of $\Gamma(Q)$ and $\Delta(Q)$ as directed graphs with the direction of edge inherited from the quiver.
Galashin and Pylyavskyy have shown that a bipartite quiver $Q$ is recurrent if and only if $A_{\Gamma(Q)}$ and $A_{\Delta(Q)}$ commute \cite[Corollary 2.3]{GP1}.
A bigraph $(\Gamma, \Delta)$ is called an \emph{admissible $ADE$ bigraph} if every component of both $\Gamma$ and $\Delta$ is an $ADE$ Dynkin diagram and the adjacency matrices of $\Gamma$ and $\Delta$ commute.
In the case of an admissible $ADE$ bigraph, each connected component of $\Gamma$, and similarly of $\Delta$, will be an $ADE$ Dynkin diagram will the same Coxter number~\cite[Corollary 4.4]{Stem}.
More generally, we wish to also consider what we will refer to as \emph{half-finite} bigraphs where for at least one of $\Gamma$ or $\Delta$ each connected component is a $ADE$ Dynkin diagram.
Note the half-finite case includes both the admissible $ADE$ bigraph case (which are exactly those quivers which are Zamolodchikov periodic~\cite{GP1}) as well as the \emph{affine $\boxtimes$ finite} case in the classification of Galashin and Pylyavskyy~\cite{GP2}.
An example of a bipartite recurrent quiver is shown in Figure~\ref{fig:square}.
Let $Q$ denote the bipartite recurrent quiver in Figure~\ref{fig:square}.
The edges of $\Gamma(Q)$ correspond to the thick red arrows while the edges of $\Delta(Q)$ correspond to the thin blue arrows.

For an $ADE$ Dynkin diagram $\Lambda$ we denote its Coxeter number by $h(\Lambda)$ and its number of positive roots by $|\Phi_+(\Lambda)|$.
These quantities will be important in the maximal green sequences we construct.
Table~\ref{tbl:hroots} shows the values for $h(\Lambda)$ and $|\Phi_+(\Lambda)|$ for each $ADE$ Dynkin diagram $\Lambda$.
We now present a result due to Galashin and Pylyavskyy generalizing the result for admissible $ADE$ bigraphs.

\begin{lemma}[{\cite[Corollary 1.1.9]{GP2}}]
If $(\Gamma, \Delta)$ is a half-finite bigraph so that each component of $\Gamma$ is an $ADE$ Dynkin diagram, then the Coxeter number of each component of $\Gamma$ will be the same.
\label{lem:h}
\end{lemma}

If $Q$ is an orientation of an $ADE$ Dynkin diagram $\Gamma$, then the length of the longest possible maximal green sequence is $|\Phi_+(\Lambda)|$ which has been shown in~\cite[Theorem 4.4]{BDP} and~\cite[Proposition 7.3]{Qiu}.
A quiver $Q$ is an \emph{alternating orientation} of an $ADE$ Dynkin diagram $\Lambda$ if it is an orientation of $\Lambda$ so that every vertex is either a source or sink.
In the case we have an alternating orientation, we will be interested in a certain maximal green sequence of length $|\Phi_+(\Lambda)|$ coming from bipartite dynamics.
We may assume we have a bipartition of $Q$ such that all sinks are the white vertices and all sources are the black vertices.
The maximal green sequence in the following lemma was first observed by Keller~\cite{KellerSquare}.

\begin{lemma}[\cite{KellerSquare}]
Let $Q$ be an alternating orientation of an $ADE$ Dynkin diagram with Coxeter number $h$.
If $h = 2k$, then $(\mu_{\bullet} \mu_{\circ})^k$ is a maximal green sequence.
If $h = 2k+1$, then $\mu_{\circ}(\mu_{\bullet}\mu_{\circ})^k$ is a maximal green sequence.
\label{lem:ADEMGS}
\end{lemma}

\begin{table}
\[
\begin{array}{|c|c|c|c|c|c|}
\hline
\Lambda & A_n & D_n & E_6 & E_7 & E_8  \\ \hline
h(\Lambda) & n+1 & 2n-2 & 12 & 18 & 30 \\ \hline 
|\Phi_+(\Lambda)| & \binom{n+1}{2} & n^2 - n & 36 & 63 & 120 \\ \hline
\end{array}
\]
\caption{Coxeter numbers and number of positive roots for $ADE$ types.}
\label{tbl:hroots}
\end{table}

We are ready to state and prove our theorem which gives a maximal green sequence for any half-finite bipartite recurrent quiver.
Notice the assumption that $\Gamma(Q)$ consists of connected components which are all $ADE$ Dynkin diagrams can easily be exchanged for the assumption that $\Delta(Q)$ consists of connected components which are all $ADE$ Dynkin diagrams.
Also the assumption on white vertices is only to allow us to explicitly state the maximal green sequences.
An easy modification gives the correct statement of the theorem with the roles of black and white vertices reversed.

\begin{theorem}
Let $Q$ be a half-finite bipartite recurrent quiver.
Assume that $\Gamma(Q)$ consists of connected components which are all $ADE$ Dynkin diagrams.
Further assume that with the orientation induced by $Q$ the white vertices are sinks in $\Gamma(Q)$ and sources is $\Delta(Q)$.
Let $h$ be the Coxeter number of some component of $\Gamma(Q)$.
If $h = 2k$ is even, then $(\mu_{\bullet} \mu_{\circ})^k$ is a maximal green sequence of $Q$.
If $h = 2k+1$ is odd, then $\mu_{\circ}(\mu_{\bullet}\mu_{\circ})^k$ is a maximal green sequence of $Q$.
\label{thm:recurrentMGS}
\end{theorem}

\begin{proof}
We will construct a maximal green sequence for $Q$ via component preserving mutations where components are given by the connected components of $\Gamma(Q)$.
By construction within each component every vertex will be either a source or sink.
Under our assumptions white vertices are initially sinks while black vertices are initially sources within each component.
Since $Q$ is a bipartite recurrent quiver $\mu_{\circ}(Q) = Q^{op}$ and $\mu_{\bullet}(\mu_{\circ}(Q)) = Q$.
Initially, mutation at any white vertex will be component preserving as each white vertex is a sink within its component and thus all arrows to other components will be outgoing.
Mutation at a given white vertex will not change the fact another white vertex is component preserving.
For the same reason mutation at any black vertex is component preserving in $Q^{op}$.
It follows that $(\mu_{\bullet} \mu_{\circ})^m$ and $\mu_{\circ}(\mu_{\bullet}\mu_{\circ})^m$ are component preserving sequences of mutations for any $m$.
By Lemma~\ref{lem:h} each component has the same Coxeter number.
Lemma~\ref{lem:ADEMGS} says that we do indeed have maximal green sequences on each component and therefore the theorem is proven by appealing to Corollary~\ref{cor:CPMGS}.
\end{proof}

\section{Other applications}\label{sec:other}
In this section we provide a variety of uses of the technique of component preserving mutations.

\subsection{Quantum dilogarithms}
We will review Keller's~\cite{KellerQdilog} association of a product of quantum dilogarithms with a sequence of mutations.
We will then consider properties of such products of quantum dilogarithms which come from component preserving mutations.
Let $q^{\frac{1}{2}}$ be an indeterminant.
We define the \emph{quantum dilogarithm} as
\[
\mathbb{E}(y) = 1 + \frac{q^{\frac{1}{2}}y}{q-1} + \cdots + \frac{q^{\frac{n^2}{2}}y^n}{(q^n - 1)(q^n - q) \cdots (q^n - q^{n-1})} + \cdots \]
which is consider as an element of the power series ring $\mathbb{Q}(q^{\frac{1}{2}})[[y]]$.
Keller has shown how reddening sequences give identities of quantum dilogarithms in a certain quantum algebra determined by a quiver.

Given a quiver $Q$ with vertex set $V$ and skew-symmetric adjacency matrix $B = (b_{uv})$ we obtain a lattice $\Lambda = \mathbb{Z}^V$ with basis $\{e_v\}_{v \in V}$.
There is a skew-symmetric bilinear form $\lambda: \Lambda \times \Lambda \to \mathbb{Z}$ defined by
\[\lambda(e_u, e_v) := b_{uv}.\]
The \emph{completed quantum algebra}  of the quiver $Q$, denoted by $\widehat{\mathbb{A}}_Q$, is then the noncommutative power series ring modulo relations defined as
\[\widehat{\mathbb{A}}_Q := \mathbb{Q}(q^{\frac{1}{2}})\langle\langle y^{\alpha}, \alpha \in \Lambda: y^{\alpha}y^{\beta} = q^{\frac{1}{2}\lambda(\alpha, \beta)}y^{\alpha+\beta}\rangle\rangle.\]
For any sequence  $\sigma = (i_1, i_2, \dots, i_N)$ of vertices in $Q$ we define
\[Q_{\sigma, t} := \mu_{i_t} \circ \mu_{i_{t-1}} \circ \cdots \circ \mu_{i_1} (Q)\]
for $0 \leq t \leq N$ where $Q_{\sigma, 0} = Q$.
We then define the product  $\mathbb{E}_{Q,\sigma} \in \widehat{\mathbb{A}}_Q$ as
\[\mathbb{E}_{Q,\sigma} := \mathbb{E}(y^{\epsilon_1 \beta_1})^{\epsilon_1}\mathbb{E}(y^{\epsilon_2 \beta_2})^{\epsilon_2} \cdots \mathbb{E}(y^{\epsilon_N \beta_N})^{\epsilon_N}\]
where $\beta_t$ is the $c$-vector corresponding to vertex $i_t$ in $Q_{\sigma, t-1}$ and $\epsilon_t \in \{\pm 1\}$ is the common sign on the entries of $\beta_t$.
If $\sigma$ is a reddening sequence, then $\mathbb{E}_{Q,\sigma}$ is known as the \emph{combinatorial Donaldson-Thomas invariant} of the quiver $Q$.
If $\sigma$ and $\sigma'$ are two reddening sequences, then we have the quantum dilogarithm identity $\mathbb{E}_{Q,\sigma} = \mathbb{E}_{Q,\sigma'}$~\cite[Theorem 6.5]{KellerSquare}.

\begin{figure}
\begin{tikzpicture}
\node (1) at (-1,0) {$1$};
\node (2) at (0,0) {$2$};
\node (3) at (1,0) {$3$};
\draw[-{latex}] (2) to (1);
\draw[-{latex}] (2) to (3);
\end{tikzpicture}
\caption{An alternating orientation of the Dynkin diagram $A_3$.}
\label{fig:altA3}
\end{figure}

In the case that $\alpha = \sum_{i \in I} e_i$ where $I = \{i_1, i_2, \dots, i_{\ell}\}$ we may write $y_{i_1i_2\cdots i_{\ell}\\}$ in place of $y^{\alpha}$.
Using this abbreviated notation, the well known \emph{pentagon identity} is 
\begin{equation}
\E(y_1)\E(y_2) = \E(y_2)\E(y_{12})\E(y_1)
\label{eq:penta}
\end{equation}
and can be seen by looking at the two maximal green sequences for the quiver $Q = (1 \to 2)$.
Now consider the quiver in Figure~\ref{fig:altA3} which is an alternating orientation of the Dynkin diagram $A_3$.
The two maximal green sequences
\[(2,1,3)\]
and
\[(1,3,2,1,3,2)\]
give the quantum dilogarithm identity
\begin{equation}
\E(y_2)\E(y_1)\E(y_3) = \E(y_1)\E(y_3)\E(y_{123})\E(y_{23})\E(y_{12})\E(y_2).
\label{eq:altA3}
\end{equation}
Reineke~\cite{Rei} has given quantum dilogarithm identities associated to any alternating orientation of an $ADE$ Dynkin diagram which generalize Equations~(\ref{eq:penta}) and~(\ref{eq:altA3}).
Using cluster algebra theory, Keller~\cite{KellerSquare} has further generalized these identities to square products associated to pairs of $ADE$ Dynkin diagrams.
Even more general identities follow from Theorem~\ref{thm:recurrentMGS} since we have now produced two maximal green sequences for any Zamolodchikov periodic quiver. 

Let us give a few properties of quantum dilogarithm products coming from component preserving mutations.
For $\alpha = \sum_i a_i e_i \in \Lambda$ we define its \emph{support} to be $\Supp(\alpha) := \{i  : a_i \neq 0\}$.
Consider a quiver $Q$, a subset of vertices $C \subseteq V(Q)$, and a sequence of vertices $\sigma = (i_1, i_2, \dots, i_N)$.
Define $\sigma|_C$ to be the restriction of $\sigma$ to $C$ (i.e. $\sigma$ where all vertices not in $C$ have been deleted).
Again write 
\[\mathbb{E}_{Q,\sigma} = \mathbb{E}(y^{\epsilon_1 \beta_1})^{\epsilon_1}\mathbb{E}(y^{\epsilon_2 \beta_2})^{\epsilon_2} \cdots \mathbb{E}(y^{\epsilon_N \beta_N})^{\epsilon_N}\]
and define $(\E_{Q, \sigma})|_C$ to be the product $\E_{Q, \sigma}$ (taken in the same order) with the terms $\mathbb{E}(y^{\epsilon_t \beta_t})^{\epsilon_t}$ removed whenever $i_t \not\in C$.
We now provide a proposition which tells us that when a reddening sequence of component preserving mutations is performed, there is a restriction on the support of the $c$-vectors occurring in the combinatorial Donaldson-Thomas invariant.
The proposition follows readily from the definitions and Remark~\ref{support}.
When $\pi$ is a set partition of a set $X$ and $x \in X$ is an element of that set, we will use $\pi(x)$ to denote the block of the set partition $\pi$ which contains $x$.

\begin{proposition}
Let $(Q, \pi)$ be a partitioned quiver so that $\sigma = (i_1, i_2, \dots, i_N)$ is a component preserving sequence of vertices.
If $C = Q_j$ is some component, then $\mathbb{E}_{Q,\sigma|_C} = (\E_{Q, \sigma})|_C$.
Moreover, we have that $\Supp(\beta_t) \subseteq \pi(i_t)$ for each $1 \leq t \leq N$.
\label{prop:support}
\end{proposition}
When $Q$ is such that $(\Gamma(Q), \Delta(Q))$ is an admissible $ADE$ bigraph we can obtain a second maximal green sequence from Theorem~\ref{thm:recurrentMGS} by exchanging the roles of $\Gamma(Q)$ and $\Delta(Q)$.
A square product of two $ADE$ Dynkin diagrams produces a quiver $Q$ such that $(\Gamma(Q), \Delta(Q))$ is an admissible $ADE$ bigraph.
For square products of $ADE$ Dynkin diagrams Keller~\cite{KellerSquare} has previously produced the maximal green sequences in Theorem~\ref{thm:recurrentMGS}.
The square product of $A_3$ and $A_4$ is shown in Figure~\ref{fig:square}.
Stembridge's classification~\cite{Stem} of admissible $ADE$ bigraphs includes more than just those bigraphs encoding square products of $ADE$ Dynkin diagrams.
Thus, Theorem~\ref{thm:recurrentMGS} provides new quantum dilogarithm identites which can be thought of as generalizations of the pentagon identity.
An infinite family examples of quivers which are not square products are the \emph{twists} of an $ADE$ Dynkin diagrams~\cite[Example 1.4]{Stem}.
The quiver $Q$ which is the twist of $A_3$ is shown in Figure~\ref{fig:twist}.
On the left of Figure~\ref{fig:twist} the quiver is pictured to indicated the bigraph $(\Gamma(Q), \Delta(Q))$, and on the right we show the quiver with vertex labels.
The two expressions of the combinatorial Donaldson-Thomas invariant of $Q$ obtain from the maximal green sequences constructed in Theorem~\ref{thm:recurrentMGS} are
\begin{equation}
\E(y_1)\E(y_3)\E(y_4)\E(y_6)\E(y_{123})\E(y_{456})\E(y_{23})\E(y_{12})\E(y_{56})\E(y_{45})\E(y_2)\E(y_4)
\label{eq:dilog1}
\end{equation}
and
\begin{equation}
\E(y_2)\E(y_5)\E(y_{15})\E(y_{35})\E(y_{24})\E(y_{26})\E(y_{246})\E(y_{135})\E(y_1)\E(y_3)\E(y_4)\E(y_6).
\label{eq:dilog2}
\end{equation}
These expressions are equal and give one example of the quantum dilogarithm identities obtained from Theorem~\ref{thm:recurrentMGS}.
Looking at supports we can verify Proposition~\ref{prop:support} in this example.
Expression~(\ref{eq:dilog1}) comes from considering $\{1,2,3\}$ and $\{4,5,6\}$ as components while Expression~(\ref{eq:dilog2}) comes from considering $\{1,3,5\}$ and $\{2,4,6\}$ as components.
The maximal green sequences corresponding to the  products of quantum dilogarithms in Equations~(\ref{eq:dilog1}) and~(\ref{eq:dilog2}) are
\[(1,3,4,6,2,5,1,3,4,6,2,5)\]
and
\[(2,5,1,3,4,6,2,5,1,3,4,6)\]
respectively.

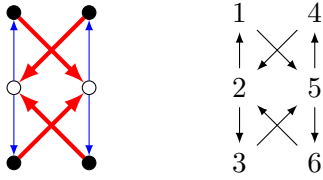
\begin{figure}
\begin{tikzpicture}
\node[draw, circle,fill=black,scale=0.5] (1) at (-0.5,1) {};
\node[draw, circle,fill=white,scale=0.5] (2) at (-0.5,0) {};
\node[draw, circle,fill=black,scale=0.5] (3) at (-0.5,-1) {};
\node[draw, circle,fill=black,scale=0.5] (4) at (0.5,1) {};
\node[draw, circle,fill=white,scale=0.5] (5) at (0.5,0) {};
\node[draw, circle,fill=black,scale=0.5] (6) at (0.5,-1) {};

\draw[-{latex}, blue] (2) to (1);
\draw[-{latex}, blue] (2) to (3);
\draw[-{latex}, blue] (5) to (4);
\draw[-{latex}, blue] (5) to (6);
\draw[-{latex}, red, ultra thick] (4) to (2);
\draw[-{latex}, red, ultra thick] (6) to (2);
\draw[-{latex}, red, ultra thick] (1) to (5);
\draw[-{latex}, red, ultra thick] (3) to (5);

\node (1) at (2.5,1) {$1$};
\node (2) at (2.5,0) {$2$};
\node (3) at (2.5,-1) {$3$};
\node (4) at (3.5,1) {$4$};
\node (5) at (3.5,0) {$5$};
\node (6) at (3.5,-1) {$6$};

\draw[-{latex}] (2) to (1);
\draw[-{latex}] (2) to (3);
\draw[-{latex}] (5) to (4);
\draw[-{latex}] (5) to (6);
\draw[-{latex}] (4) to (2);
\draw[-{latex}] (6) to (2);
\draw[-{latex}] (1) to (5);
\draw[-{latex}] (3) to (5);
\end{tikzpicture}
\caption{The quiver obtained from the twist of $A_3$.}
\label{fig:twist}
\end{figure}

\subsection{Minimal length maximal green sequences}
There has been recent interest in finding maximal green sequences of minimal possible length for a given quiver~\cite{CDRS,GMS}.
We will now show how minimal length maximal green sequences can be constructed with component preserving mutations.
In additional to being a natural question to ask about maximal green sequences, it has been observed by Garver, McConville, and Serhiyenko that the minimal possible length of a maximal green sequence may be related to derived equivalence of cluster tilted algebras (see ~\cite[Question 10.1]{GMS}).
The following result is a component preserving generalization of~\cite[Proposition 4.4]{GMS} which considers the direct sum case.

\begin{lemma}
Let $(Q, \pi)$ be a partitioned quiver with $\pi = \pi_1 / \pi_2 / \cdots / \pi_{\ell}$.
Also let $\sigma_i$ be a minimal length maximal green sequence for $Q_i$ for each $1 \leq i \leq \ell$.
If $\tau$ is a component preserving shuffle of $\sigma_1, \sigma_2, \dots, \sigma_n$, then $\tau$ is a minimal length maximal green sequence for $Q$.
\label{lem:min}
\end{lemma}

\begin{proof}
Let $L_i$ be the length of a minimal length maximal green sequence of $Q_i$ for each $1 \leq i \leq \ell$ and let $L = L_1 + L_2 + \cdots L_{\ell}$.
By Corollary~\ref{cor:CPMGS} we know that $\tau$ is a maximal green sequence and will have length $L$.
So, we now need to show that there are no shorter maximal green sequences.
Consider any maximal green sequence $\tau'$ for $Q$.
By~\cite[Theorem 3.3]{GMS} it follows that for each $1 \leq i \leq \ell$ there is a subsequence of mutations in $\tau'$ at vertices in $Q_i$ which is a maximal green sequence of $Q_i$.
This means $\tau'$ must mutate at vertices of $Q_i$ at least $L_i$ times for each $1 \leq i \leq \ell$.
Since $\pi$ is a partition, $Q_i$ and $Q_j$ share no vertices when $i \neq j$.
It follows that $\tau'$ has length at least $L = L_1 + L_2 + \cdots + L_{\ell}$.
\end{proof}

To illustrate a use of Lemma~\ref{lem:min}, let $Q$ be the quiver\footnote{The use of Lemma~\ref{lem:min} readily generalizes to quivers similar to $Q$ with longer cycle or longer path.} in Figure~\ref{fig:min}.
We will take the set partition $\{v_1, v_2, v_3, v_4, v_5\}/\{u_1, u_2, u_3, u_4\}$.
A minimal length maximal green sequence for $Q$ is then
\[(u_1, u_2, u_3, v_1, v_2, v_3, v_4, v_5, v_3, v_2, v_1, u_4)\]
which is a shuffle of $(v_1, v_2, v_3, v_4, v_5, v_3, v_2, v_1)$ and $(u_1, u_2, u_3,u_4)$.
The first is a maximal green sequence for the cycle by~\cite[Lemma 4.2]{Bucher} and is of minimal length by~\cite[Theorem 6.1]{GMS}.
The second is a maximal green sequence coming from an admissible numbering by sources.

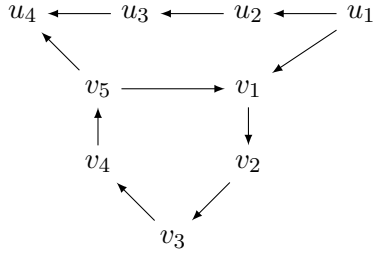
\begin{figure}
\begin{tikzpicture}
\node (1) at (1,0) {$v_1$};
\node (2) at (1,-1) {$v_2$};
\node (3) at (0,-2) {$v_3$};
\node (4) at (-1,-1) {$v_4$};
\node (5) at (-1,0) {$v_5$};
\draw[-{latex}] (1) to (2);
\draw[-{latex}] (2) to (3);
\draw[-{latex}] (3) to (4);
\draw[-{latex}] (4) to (5);
\draw[-{latex}] (5) to (1);

\node (6) at (2.5,1) {$u_1$};
\node (7) at (1,1) {$u_2$};
\node (8) at (-0.5,1) {$u_3$};
\node (9) at (-2,1) {$u_4$};
\draw[-{latex}] (6) to (7);
\draw[-{latex}] (7) to (8);
\draw[-{latex}] (8) to (9);
\draw[-{latex}] (6) to (1);
\draw[-{latex}] (5) to (9);
\end{tikzpicture}
\caption{A quiver where a minimal length maximal green sequence can be found by component preserving mutations.}
\label{fig:min}
\end{figure}

\subsection{Exponentially many maximal green sequences for Dynkin quivers}
In \cite[Remark 4.2 (3)]{BDP} the authors observe that the number of maximal green sequences of the lineary oriented Dynkin quiver of type $A_n$ seems to grow exponentially with $n$.
The main result of this section will affirm this observation.
A \emph{Dynkin quiver of type $A_n$} is any orientation of the Dynkin diagram of type $A_n$.
The \emph{linearly oriented Dynkin quiver of type $A_n$} has vertex set $\{i : 1 \leq i \leq n\}$ and arrow set $\{i \to i+1 : 1 \leq i < n\}$.
Figure~\ref{fig:A5} shows the linearly oriented Dynkin quiver of type $A_5$.
We will show that the number of maximal green sequences of arbitrarily oriented Dynkin quiver of type $A_n$ is at least expontential.
We give a simple and explicit proof of an exponential lower bound to $|\green(Q)|$ where $Q$ is any Dynkin quiver of type $A_n$.
After we will provide an improved bound in the case $Q$ is a linearly oriented Dynkin quiver of type $A_n$.

Recall the Fibonacci numbers are defined by the recurrence $F_1 =1$, $F_2 =2$, and $F_n = F_{n-1} + F_{n-2}$ for $n \geq 2$.
A closed form expression for $F_n$ is
\[F_n = \frac{\phi^n - \psi^n}{\sqrt{5}}\]
where
\begin{align*}
    \phi &= \frac{1 + \sqrt{5}}{2} & \psi &= \frac{1 - \sqrt{5}}{2}.
\end{align*}

\begin{figure}
    \centering
    \begin{tikzpicture}
    \node at (0,0) (1) {$1$};
    \node at (1,0) (2) {$2$};
    \node at (2,0) (3) {$3$};
    \node at (3,0) (4) {$4$};
    \node at (4,0) (5) {$5$};
    \draw[-{latex}] (1) to (2);
    \draw[-{latex}] (2) to (3);
    \draw[-{latex}] (3) to (4);
    \draw[-{latex}] (4) to (5);
    \end{tikzpicture}
    \caption{The linearly oriented Dynkin quiver $A_5$.}
    \label{fig:A5}
\end{figure}

\begin{proposition}
If $Q$ is a Dynkin quiver of type $A_n$ for any $n \geq 1$, then $|\green(Q)| \geq F_{n+1}$.
\label{prop:AnFn}
\end{proposition}
\begin{proof}
It can be easily checked that $|\green(Q)| = 1 = F_2$ for $n=1$ and $|\green(Q)| = 2 = F_3$ for $n = 2$.
For $n \geq 3$ assume inductively that $|\green(Q)| \geq F_{m+1}$ for all $1 \leq m < n$.
We first consider components of $Q$ coming from the set partition $C/C'$ where $C = \{i : 1 \leq i \leq n-1\}$ and $C' =  \{n\}$.
Here $Q$ is isomorphic to a direct sum of a Dynkin quiver of type $A_{n-1}$ and a Dynkin quiver of type $A_1$.
Hence, $Q$ has at least $|\green(Q|_C)|$ maximal green sequences by considering any maximal green sequence on $Q|_C$ with $(n)$ either appended or prepened depending of whether $(n-1) \to n \in Q$ or $n \to (n-1) \in Q$.

Next consider components of $Q$ coming from the set partition $D/D'$ where $D= \{i : 1 \leq i \leq n-2\}$ and $D' =  \{n-1, n\}$.
Now $Q$ is isomorphic to a direct sum of Dynkin quiver of type $A_{n-2}$ and a Dynkin quiver of type $A_2$.
Thus, $Q$ has at least $|\green(Q|_D)|$ maximal green sequences by considering any maximal green sequence on $D$ with:
\begin{itemize}
    \item $(n,n-1,n)$ appended if $(n-2) \to (n-1), (n-1) \to n \in Q$.
    \item $(n,n-1,n)$ prepended if $(n-1) \to (n-2), (n-1) \to n \in Q$.
    \item $(n-1,n,n-1)$ appended if $(n-2) \to (n-1), n \to (n-1) \in Q$.
    \item $(n-1,n,n-1)$ prepended if $(n-1) \to (n-2), n \to (n-1) \in Q$.
\end{itemize}
We see that the set of maximal green sequences for $Q$ coming from $\green(Q|_C)$ are disjoint from those coming from $\green(Q|_D)$.
In the former $n$ is mutated at only once and is either mutated first or last in the sequence.
In the latter $n$ is either mutated at twice or otherwise is neither the first nor the last mutation.
It follows that
\[|\green(Q)| \geq |\green(Q|_C)| + |\green(Q|_D)| \geq F_{n} + F_{n-1} = F_{n+1}\]
and the proposition is proven.
\end{proof}

For a linearly oriented Dynkin quiver $Q$ of type $A_n$, we have the maximal green sequence
\[(n,n-1, \cdots 1, n, n-1, \dots, 2, \dots, n,n-1,n)\]
which we will call the \emph{long sequence}\footnote{There are many possible maximal green sequences of this maximal length. So, we should perhaps say \emph{a} long sequence instead of \emph{the} long sequence. However, we wish to emphasize that in this section we will be using only this particular sequence of mutations.}.
As an example in the case $n=4$ the long seqeunce is
\[(4,3,2,1,4,3,2,4,3,4).\]
The long sequence is a maximal green sequence coming from a reduced factorization of the longest element in the corresponding Coxeter group.

\begin{proposition}
If $Q$ is the linearly oriented Dynkin quiver of type $A_n$ for any $n \geq 1$, then $|\green(Q)| \geq 2^{n-1}$.
\label{prop:An2n}
\end{proposition}
\begin{proof}
For $n=1$ we have $|\green(Q)| = 1$ and for $n=2$ and $|\green(Q)| = 2$.
Given $n \geq 3$, assume inductively that $|\green(Q)| \geq 2^{m-1}$ for all $1 \leq m < n$.
Consider components from the set partition $C^{(k)}/D^{(k)}$ where $C^{(k)} = \{1,2,\dots k\}$ and $D^{(k)} = \{k+1,k+2, \dots, n\}$ for $0 \leq k < n$.
For each $k$, our quiver $Q$ has at least $|\green(Q|_{C^{(k)}})|$ many maximal green sequences by appending the long sequence of $Q|_{D^{(k)}}$ to any maximal green sequence of $Q|_{C^{(k)}}$.
Here we count one maximal green sequence, the long sequence for $Q$, when $k=0$.
In the long sequence for $Q|_{D^{(k)}}$ vertex $n$ is mutated at $n-k$ times, and thus the maximal green sequences coming from $\green(Q|_{C^{(k_1)}})$ and $\green(Q|_{C^{(k_2)}})$ are disjoint for $k_1 \neq k_2$.
So,
\[|\green(Q)| \geq \sum_{k=0}^{n-1}|\green(Q|_{C^{(k)}})| \geq 1 + \sum_{k=1}^{n-1} 2^{k-1} = 2^{n-1}\]
and the proposition follows.
\end{proof}

Let $\green(A_n)$ denote the set of maximal green sequences of a linearly oriented type $A_n$ quiver.
Proposition~\ref{prop:An2n} is constructive starting from knowing $\green(A_1) = \{(1)\}$ and $\green(A_2) = \{(1,2), (2,1,2)\}$.
The method in the proof of Proposition~\ref{prop:An2n} produces 
\[\{(1,2,3),(2,1,2,3), (1,3,2,3),(3,2,1,3,2,3)\} \subseteq \green(A_3),\]
and we show in Table~\ref{tab:A4} the $8$ maximal green sequences in $\green(A_4)$ constructed by applying the proof of Proposition~\ref{prop:An2n} one more time.
The maximal green sequences in Table~\ref{tab:A4} are arranged according to the set partition $C^{(k)}/D^{(k)}$.
\begin{table}[]
    \centering
    \[
    \begin{array}{|c|c|} \hline
         k& \text{Maximal green sequences}  \\ \hline
         0& (4,3,2,1,4,3,2,4,3,4)\\ \hline
         1& (1,4,3,2,4,3,4)\\ \hline
         2& (1,2,4,3,4), (2,1,2,4,3,4)\\ \hline
         3& (1,2,3,4),(2,1,2,3,4),(1,3,2,3,4),(3,2,1,3,2,3,4) \\ \hline
    \end{array}
    \]
    \caption{Maximal green sequences in $\green(A_4)$ constructed in proof of Proposition~\ref{prop:An2n} according to set partition $C^{(k)}/D^{(k)}$.}
    \label{tab:A4}
\end{table}

\bibliographystyle{alpha}
\bibliography{refs}

\end{document}